\documentclass[a4paper,11pt]{article}
\usepackage{multirow}
\usepackage{mathtools}     
\usepackage{subcaption}
\usepackage{cases}
\usepackage{amsfonts,amssymb,amsmath,amsthm,latexsym,bbm}
\usepackage{enumerate}
\usepackage{epsf,epsfig}
\usepackage{xcolor,colortbl,color}
\usepackage{graphicx,graphics}
\usepackage{caption}
\usepackage[utf8]{inputenc}   
\usepackage{tikz}
\usepackage[colorlinks, citecolor=green!50!black,linkcolor=blue,urlcolor=blue]{hyperref}  
\usepackage[english]{babel} 
\mathtoolsset{showonlyrefs}
\makeatletter
\def\namedlabel#1#2{\begingroup
    #2%
    \def\@currentlabel{#2}%
    \phantomsection\label{#1}\endgroup
}
\makeatother
\newtheorem{example}{}

\def\chv#1{\{\,#1\,\}}
\def\abs#1{\left\vert #1 \right\vert} 
\def\crt#1{\left[ #1 \right]} 
\def\prt#1{\left( #1 \right)} 
\def\Ind#1{\mathbbm{1}_{#1}}  
\def\x{m}
\newcommand{\gep}{\varepsilon} 
\newcommand{\Rc}{\bar{\mathbb{R}}_+} 
\newcommand{\Pc}{\bar{\mathcal{P}}}
\newcommand{\CL}{C_b}
\newcommand{\Kk}{K}
\newcommand{\pd}[2]{\langle#1, #2\rangle}

\numberwithin{equation}{section}
\theoremstyle{plain}
\newtheorem{thm}{Theorem}[section]
\newtheorem{corollary}{Corollary}[section]
\newtheorem{proposition}{Proposition}[section]
\newtheorem{lemma}{Lemma}[section]
\theoremstyle{remark}
\newtheorem{remark}{Remark}[section]
\newtheorem{definition}{Definition}[section]

\title{Laws of large numbers
  for weighted sums of independent random variables:
a game of mass}

\author{
  
  Luca Avena\footnotemark[1] \,,
  Conrado da Costa\footnotemark[2]
}
\begin{document}

\maketitle

\footnotetext[1]{Mathematical Institute, Leiden University, 
P.O.\ Box 9512, 2300 RA Leiden, The Netherlands}
\footnotetext[2]{
Mathematical Sciences \& Computer Science Building,
Upper Mountjoy Campus,
Stockton Road,
Durham University,
DH1 3LE,
Durham, UK.
}

\begin{abstract}  
We consider weighted sums of independent
random variables regulated by an increment sequence
and provide operative conditions that ensure
strong law of large numbers for such sums to hold
in both the centred and non-centred case.
In the literature, the existing criteria
for the strong law are either implicit or
based on restrictions on the increment sequence.
In our set up  we allow for arbitrary sequence of
increments, possibly random, provided
the random variables regulated by such
increments satisfy some mild concentration conditions.
In the non-centred case, convergence can be translated
into the behavior of a deterministic sequence and
it becomes a game of mass provided the expectation
of the random variables is a function of the
increments.
We show how different limiting scenarios can emerge
by identifying several classes of increments,
for which concrete examples will be offered. 
\end{abstract}


\section{Setup, literature and overview}

Let $\mathbb{X} := \chv{X_k, k \in \mathbb{N}}$
be a sequence of  independent real valued
random variables with finite mean and
$\mathbb{A} = (a_{n,k}\in \mathbb{R}_+ ; n,k\in\mathbb{N})$
a \emph{Toeplitz summation matrix},
i.e., $\mathbb{A}$ satisfies
\begin{align}
\label{vanish_fixed_k} \lim_{n} a_{n,k} = 0,\\
\label{limitsum1} \lim_{n} \sum_{k} a_{n,k} = 1,\\
\label{bound_absolute_value}  \sup_{n} \sum_{k} \abs{a_{n,k}}< \infty.
\end{align}
In this set up, one seeks conditions
on $\mathbb{X}$ and $\mathbb{A}$
to ensure convergence in probability
or almost sure convergence for the sequence
$\chv{S_n,n \in \mathbb{N}}$, where
\begin{equation}\label{sum}
S_n : = \sum_{k} a_{n,k} X_k.
\end{equation}

This type of questions,
known in the literature as weak/strong
Law of Large Numbers (LLN)
have been investigated since the birth
of probability theory, see~\cite{Ber1713},
and 
has been extensively studied in the XX century,
for different summation methods, e.g. Voronoi sums,
Beurling moving averages see~\cite{Bin17, Bin15},
see also~\cite{Rev67, Kle14}
and references therein for a classical
and a more recent account.
The quest for operative conditions that apply
to a wide range of $(\mathbb{X},\mathbb{A})$
and ensure weak/strong convergence of $S_n$
has been the the subject
of~\cite{JamOrePru65, Pru66, Sto68, Roh71}.

When the elements of $\mathbb{X}$ are i.i.d.
mean zero random variables, the \emph{weak}
LLN is equivalent to
$\lim_n \max_k a_{n,k} = 0$, 
see ~\cite[Theorem 1]{Pru66}.
In \cite[Theorem 2]{Pru66},
the following sufficient conditions
for the \emph{strong} LLN are given: 

\begin{equation}\label{PRUiid}
  \mathbb{E}[X_1^{1 + \frac{1}{\gamma}}]
  <\infty \quad\text{ and }\quad
  \limsup_{n} n^{\gamma }\max_{k} a_{n,k}< \infty,
  \quad\quad \text{ for some } \gamma>0.
\end{equation}

For (mean-zero) independent but not
identically distributed variables,
similar sufficient conditions have
been examined in~\cite{JamOrePru65, Sto68,Roh71}.
In particular, in analogy with the
two conditions in~\eqref{PRUiid},
these references require that the
variables $X_k$'s are stochastically
dominated by a random variable $X_*$
satisfying a moment condition,
and that the associated coefficients
$a_{n,k}$ decay sufficiently fast.
     
Unlike the above references, in this paper
we impose concentration conditions on $\mathbb{X}$
and obtain sufficient conditions for
the weak/strong LLN
when $\limsup_n \max_k a_{k,n}>0$.
Here, as in~\cite{JamOrePru65},
we consider a family of weights,
which we will refer to as \emph{masses},
$\mathbf{m}:=\chv{m_k \in \mathbb{R}_+ , k  \in \mathbb{N}}$.
We assume that the mass sequence $\mathbf{m}$
is such that
\begin{equation}\label{divergent}
\sum_{k \in \mathbb{N}}m_k = \infty.
\end{equation}
Set $M_n : = \sum_{k = 1}^n m_k$ and 
\begin{equation}\label{ToplitzA}
  a_{n,k}: = \begin{cases} \frac{m_k}{M_n}
    & \text{if } k \leq n,\\
0& \text{otherwise}.
\end{cases}
\end{equation}
Conditions~\eqref{limitsum1}
and~\eqref{bound_absolute_value}
hold true by definition.
Also, because~\eqref{divergent}
implies $\lim_n M_n = \infty$ it
follows that \eqref{vanish_fixed_k} is in force and
therefore  $\mathbb{A}$ is a Toeplitz summation
matrix.
We notice in particular that if the sum in
\eqref{divergent} is finite, then no LLN can be expected.
In
fact, if the random variables are not all
constant,
the limit random variable will have finite yet
strictly positive variance,
what precludes convergence to
a constant.
To describe our results we depart from the set up
of~\cite{JamOrePru65} and consider $X_k = X_k(m)$ to be a one parameter family of
random variables.

\noindent{\bf Our first contribution:}
The main goal of this paper is to provide (near)
optimal operative conditions on $\mathbb{X}$
to ensure that for \emph{any} sequence of positive
masses $\mathbf{m} \in \mathbb{R}_+^{\mathbb{N}}$, 
\begin{equation}\label{Sn}
S_n = S_n (\mathbf{m}):= \sum_{k = 1}^n \frac{m_k}{M_n} X_k(m_k) 
\end{equation}
converges to zero as $n \to \infty$ both in a weak and
in a strong sense, going beyond the fast coefficient
decay assumptions made in the existing literature. 
Due to the nature of the coefficients
in~\eqref{ToplitzA} we will refer to the sum
in~\eqref{Sn} as \emph{incremental sum}.

{\bf Starting motivation and applications
  in random media:}
Our original motivation to look at this type of
incremental sums  
came from the analysis
pursued in~\cite{AdH19, ACdCdH19,ACdCdH20}
of the asymptotic speed, and related large deviations, 
of a Random Walk (RW) in a dynamic random media,
referred to as Cooling Random Environment
(CRE).
This model is obtained as a perturbation of another
process, the well-known RWRE, by adding independence
through resetting. 
RWCRE, denoted by $(\overline{Z}_n)_{n\in\mathbb{N}_0}$,
is a patchwork of independent RWRE displacements
over different time intervals.
More precisely, the classical RWRE
consists of a random walk $(Z_n)_{n\in\mathbb{N}_0}$ on
$\mathbb{Z}$ with random transition kernel
given by a random environment
sampled from some law $\mu$ at time
zero. To build the dynamic transition kernel
of RWCRE we fix a partition of time
into disjoint intervals
$\mathbb{N}_0 = \bigcup_{k\in\mathbb{N}} I_k$
then, we sample a sequence of environments from
a given law $\mu$ and assign
them to the time intervals $I_k$.
To obtain the sum in~\eqref{Sn}
we let $ \abs{I_k} = m_k$ and consider $S_n = \overline{Z}_{M_n}/M_n$.
In this case,
$S_{n}$ represents the empirical speed of RWCRE at time $M_n$
and $X_k(m_k)$
represents the displacement
of an independent RWRE
$Z_{m_k}$ after $m_k$ time steps.

This type of time-perturbation of RWRE by resetting in
reality gives rise to a slightly more general sum than
the incremental one in~\eqref{Sn}. Hence we will prove
statements for the above incremental sum but also for
the more general one, referred to as
\emph{gradual sum}, as defined
in~\eqref{lastresampling}--\eqref{gradual} below. 
It is worth saying that this patchwork construction
can be naturally performed to perturb (in time or even
in space)
other models in random media, for example, to describe
polymagnets~\cite{Sulet05} based on
juxtaposition of independent Curie-Weiss
spin-systems~\cite{FriVel17} of relative sizes $m_k$'s.

A first partial analysis of the asymptotic speed for
RWCRE started in~\cite[Thm. 1.5]{AdH19}
and in~\cite[Thm. 1.12]{ACdCdH19}.
The new results here offer a full characterization of
when existence of the asymptotic averages in~\eqref{Sn}
holds
and moreover, when moving from a centered zero-mean
setup, which values can this limit attain.

\noindent{\bf Second contribution:}
Our second  goal is thus to explore in the general
non-centred case structural
conditions on the masses $\mathbf{m}$ that ensure
convergence of the weighted sums.
We will in particular identify different classes of
masses for which the resulting
limit exists and can be characterized, what we will
refer to as \emph{the game of mass}, for which many
concrete examples will be presented.
 
\noindent{\bf Structure of the paper:}
In Section~\ref{results0} we state the general LLNs
for centred random variables:
Theorems~\ref{incrementalweak} and~\ref{incrementalstrong},
respectively, for the weak and the strong laws for the
incremental sum;
Theorem~\ref{gradualstrong} for the more general
gradual sum.  
Section~\ref{game} is devoted to the game of mass were
we study concrete convergence criteria for
non-centred variables. 
A discussion on the nature of the hypotheses,
illustrated by counterexamples, is presented in
Section~\ref{hypothesis}.
Lastly, Section~\ref{proofs} contains the proofs of the
main theorems.
Appendix \ref{boundedapp} covers a technical lemma
adapted from \cite{JamOrePru65} and used in the proof
of Section \ref{sss:incremental}.
\section{LLNs for mean-zero variables} \label{results0}

In what follows all random variables are defined on a
probability space $(\Omega, \mathcal{F}, \mathbb{P})$
and $\mathbb{E}$ denotes expectation with respect to
$\mathbb{P}$.
Let  $\mathbb{X} = \chv{X_{k} (m), m \in \mathbb{R}_+,
  k \in \mathbb{N}}$ 
be a family of integrable random variables  that are
independent in $k$.
Our first statement is the the weak LLN for mean-zero
independent random variables.
\begin{thm}[{\bf Weak LLN}]
\label{incrementalweak} 
Assume that $\mathbb{X}$ satisfies the following
conditions:
\begin{description}
	\item[\namedlabel{C}{C}] \emph{(Centering)}
\begin{equation}\notag
\forall\,m \in \mathbb{R}_+, \, k \in \mathbb{N};\quad \mathbb{E}\crt{X_k(m)} = 0.
\end{equation}
\item[\namedlabel{W1}{W1}]  \emph{(Concentration)} 
\begin{equation}\notag
	\lim_{m\to\infty}    \sup_{k} \mathbb{P}\prt{\abs{X_k(m)}>\gep} =0,
	\quad  \forall \varepsilon>0. 
\end{equation}   
\item[\namedlabel{W2}{W2}]  \emph{(Uniform Integrability)} 
 \begin{equation}\notag
	\lim_{A\to \infty} \sup_{k,m} \mathbb{E} \crt{\abs{X_k(m)}\Ind{\abs{X_k(m)}>A}} = 0. 
\end{equation}
\end{description}
Then, for any sequence
$\mathbf{m} \in \mathbb{R}_+^\mathbb{N}$ that
satisfies~\eqref{divergent},
\begin{equation}\notag
\lim_{n\to\infty}\mathbb{P}\prt{\abs{S_n}>\gep}=0,\quad  \forall \varepsilon>0. 
\end{equation} 	
\end{thm}
In this centred case, as captured in the next theorem,
to obtain a strong LLN for $\chv{S_n, n \in \mathbb{N}}$ we impose further conditions on $\mathbb{X}$.
In particular the concentration condition will be
strengthen by requiring a mild polynomial decay and the
uniform integrability by a uniform domination.

\begin{thm}[{\bf Strong LLN}]
\label{incrementalstrong} 
Assume that $\mathbb{X}$ satisfies \eqref{C} and 
\begin{description}
\item[\namedlabel{S1}{S1}]\emph{(Polynomial decay)} There is a $\delta>0$ such that for all $\gep>0$ there is a $C = C(\gep)$ for which
\begin{equation}\notag
  \sup_k\mathbb{P}\left(\abs{X_k(m)}>\gep\right)<\frac{C}{m^\delta}.
\end{equation}
\item[\namedlabel{S2}{S2}] \emph{(Stochastic domination
and moment control)}  There is a random variable $X_*$
and $\gamma>0$ such that
$\mathbb{E}(\abs{X_*}^{2 +\gamma })<\infty$ and for all
$x \in \mathbb{R}$
\begin{equation}\notag
  \sup_{k,m}\mathbb{P}(X_{k}(m)>x)\leq \mathbb{P}(X_*>x). \end{equation}  
\end{description}
Then for any sequence
$\bf{m} \in \mathbb{R}_+^\mathbb{N}$
that satisfies \eqref{divergent}, 
\begin{equation}\label{sLLNincremental}
\mathbb{P}\prt{\lim_{n\to\infty}S_n= 0} = 1.
\end{equation} 
\end{thm}
As anticipated, motivated by the random walk
in~\cite{AdH19,ACdCdH19}, we next focus on a more
general sum by considering a time parameter $t$ that
runs on the positive real line partitioned into
intervals $I_k=[M_{k-1},M_k)$ of size
$m_k$: $[0,\infty)= \cup_k I_k$.
As  $t\to \infty$  the increments determined by the
partition are gradually completed as captured in
definition~\eqref{gradual} below. 
For $\bf{m} \in \mathbb{R}_+^{\mathbb{N}}$, 
let \begin{equation}\label{lastresampling}
  \ell_t = \ell_t({\mathbf{m}})
  := \inf\chv{\ell \in \mathbb{N}: M_\ell\geq t},
\end{equation} 
and set $\bar{t}: = t - M_{\ell_t - 1}$.
We define the  \emph{gradual sum}   by
\begin{equation}\label{gradual}
\mathcal{S}_t = \mathcal{S}_t
(\mathbf{m})
: = \sum_{k=1}^{\ell_t-1} \frac{m_k}{t} X_k(m_k) + 
\frac{\bar{t}}{t} X_{\ell_t}(\bar{t}). 
\end{equation}
The next theorem, is the extension of
Theorem~\ref{incrementalstrong}
to treat the gradual sum $\mathcal{S}_t$.
 \begin{thm}[{\bf Generalized strong LLN}]
\label{gradualstrong} 
Assume that $\mathbb{X}$ satisfies
\eqref{C}, \eqref{S1},
\eqref{S2}, and further:
\begin{description} 
\item[\namedlabel{S3}{S3}]{\emph{(Slow relative increment growth)}}  
for every $\gep>0$ there is a $\beta>1$, $C_\gep>0$
which for every $t,r>0$
\begin{equation}\notag
\sup_k\mathbb{P}\prt{\sup_{s \leq m} \abs{(r+s)X_{k}(r + s) - r X_{k}(r)} \geq t\gep}
\leq \frac{C_\gep m^\beta}{t^\beta}.
\end{equation}
\end{description}
Then for any sequence
$\mathbf{m} \in \mathbb{R}_+^\mathbb{N}$ that
satisfies~\eqref{divergent}
\begin{equation}
\label{e:SLLN}
 \mathbb{P}\prt{\lim_{t\to\infty}\mathcal{S}_t =0} = 1.
\end{equation} 
 \end{thm}
 \begin{remark}{\bf{(Continuity assumption for the gradual sum)}}
\label{extragradual} Assumption \eqref{S3} controls
the oscillations between the times $M_n$.
If the sequence $sX_k(s)$ was a martingale, Doob's $L^p$
inequality would yield \eqref{S3}. Also, if the
increments $$(r+s)X_{k}(r + s) - r X_{k}(r)$$ were
bounded by $f(s)$ then this condition would also follow.
Condition \eqref{S3} reveals that the one parameter
families $\chv{X_k(m),m\in \mathbb{R}}$ we consider
here possess some dependence structure or satisfy some
increment domination.
\end{remark}

If the random variables $\mathbb{X}$ are not centred,
the convergence of $S_t$ will correspond to the
convergence of $\mathbb{E}\crt{S_t}$.
This is the content of the next result.
\begin{corollary}[\textbf{Non-centred strong LLN}]
\label{gameofmass}
Assume that $\mathbb{X}$
satisfies~\eqref{S1}--\eqref{S3}
and that the sequence
$\mathbf{m}\in \mathbb{R}_+^{\mathbb{N}}$
satisfies~\eqref{divergent}.
Then, for any increasing sequence
$\chv{t_k}_{ k \in \mathbb{N}}$ with
$\lim_k t_k = \infty$ and
$\lim_k \mathbb{E}\crt{\mathcal{S}_{t_k}} =: v$,
\begin{equation}\label{e:gm}
\mathbb{P} \prt{\lim_k \mathcal{S}_{t_k} = v } = 1.
\end{equation}
\end{corollary}
We remark that it is not sufficient to examine the
convergence of the sequence $t_k = M_k$,
as the boundary term in the gradual sum may not be
negligible as, for instance, in the case where 
\begin{equation}\label{craz}
  \mathbb{E}[X_k(m)]= \frac{2^k - m}{\max\{2^k,m\}}.
\end{equation}
Indeed, for $m_k = 2^k$ $\mathbb{E}[X_k(m_k)] = 0$ and
so for $t_k = M_k$,
$\lim_k\mathbb{E}(\mathcal{S}_{t_k}) = 0$ but for
$t_k = M_{k-1} + 2^{k-1} $,
$\lim_k\mathbb{E}(\mathcal{S}_{t_k}) = 1/4$.

Interestingly, if $\mathbb{E}[X_k(m)] = v_m$ depends
only of $m$,
one can relate the convergence of $\mathcal{S}_t$ to
the structure of $\mathbf{m}$.
This is what we call \emph{the game of mass} and
explore next.

\section{The game of mass: operative conditions}
\label{game}
As a consequence of Corollary~\ref{gameofmass},
it is natural to seek for conditions on
$(\mathbb{X},\mathbf{m})$ that guarantee convergence of
the full sequence $\mathcal{S}_t$. In this section, we
assume that the expectation of $X_k(m)$ depends only
on $m$ and not on $k$, that is:
\begin{equation} \label{samemean}
\mathbb{E}\crt{X_{k}(m)} =v_m \quad \forall \, k \in \mathbb{N},
\end{equation} 
and that 
\begin{equation}\label{vinfty}
  m \mapsto v_m \quad \text{is a bounded continuous
    function in $\Rc$,}
\end{equation}
where $\Rc=[0,\infty]: = \mathbb{R}_+ \cup \{\infty\}$
be the compact metric space with the metric
$d(x,y): = \arctan(x) - \arctan(y)$.
 
We will classify the mass-sequences $\mathbf{m}$ into
two classes:
\emph{regular} and \emph{non-regular}.
The notion of regularity will be captured by the
existence of
the weak limit of the empirical measure associated to a
given mass sequence.
We start Section~\ref{regular} with the definition of
regular masses and show that,
contrary to the non-regular ones, the LLN always holds
true.
In Section~\ref{frequency}, we consider other notions
of regularities
and examine how they relate to the convergence of the
empirical measures.
Section~\ref{Exboun} is devoted to examples of bounded
masses and their
relation to the previously defined notions.
In Section~\ref{Exunboun}, we identify the regular
regime of mass sequences
that diverge.
Finally, in Section~\ref{Exran} we investigate what can
be said when the mass-sequence $\mathbf{m}$
is  \emph{random}.
The game of mass is summarized in Fig.~\ref{fig:CRE}.

\begin{figure}[htbp]
\begin{flushleft} 
\begin{tikzpicture}[scale = .95]

\draw[draw = black, line width = .3mm, fill = gray!30](-6,-3.5) rectangle (2,4);
\draw[draw = black, line width = .3mm] (2,-3.5)                  rectangle (6,4);

\draw[draw=gray!30, line width=0mm, fill=gray!30]   (0,-2.5) rectangle (5.97,-1);
\draw[draw=black, line width=.5mm]                 (-4.5,.9) rectangle (1.9,3.9);
\draw[draw=black, line width=.5mm]                 (-2.2,2.1) rectangle (.5, 1);
 \draw[draw=red,dotted, line width=.5mm]           (-4.5,-3.3) rectangle (4.6,-.5);
\draw[draw=black, dotted ,line width=.7mm]         (-4.6,-3.4) rectangle (4.7,2.3);

\node at (-2,4.5)           {Regular $\mathbf{m}$ $\qquad  \mu_t \overset{w}{\to} \mu_*$};
\node at (4.4,4.4)          {Irregular $\mathbf{m}$};
\node at (-2.7,3.6)           {Cèsaro divergent};
\node at (-1.1,1.8)         {Divergent};
\node at (-3.4,-.9)          {$\mathsf{F}_t \overset{L_1}{\to} \mathsf{F}_*$};
\node at (-3.6,2)           {$\mathsf{F}_t \overset{w}{\to} \mathsf{F}_*$};
\node[rotate=0] at (-7.3,-2.6)       {bounded $\mathbf{m}$};
\node[rotate=00] at (-7.5,1.2)          {unbounded $\mathbf{m}$};
\draw[draw=black, line width= 1mm]                 (-6.8,-1.5) -- (6,-1.5); 
\draw[draw=black, line width= 1mm]                 (2,4.8) -- (2,-3.5);

\node at (-1.5,-2.5)     {$\bullet$  \ref{a}};
\node at (-3.5,-2.5)     {$\bullet$  \ref{F0}};
\node at (-5.3,-2.5)     {$\bullet$  \ref{Fnot}};
\node at ( 5.3,-2.5)     {$\bullet$  \ref{irr}};  
\node at (2.6,-2.5 )     {$\bullet$  \ref{irregbounded}};
\node at (-1,1.3 )       {$\bullet$  \ref{mdivergent}};
\node at (1.2,1.8)       {$\bullet$  \ref{cesar}};
\node at (-3.5,3)         {$\bullet$  \ref{cesaroifreq}};
\node at (-1,0)          {$\bullet$  \ref{triangular}};
\node at (-5.3,0)        {$\bullet$  \ref{unbnofreq}};
\node at (3.5,1)         {$\bullet$  \ref{r:emp}}; 
\node at (3,-1)          {$\bullet$  \ref{irregL1unbounded}};
\node at (4,3.2)         {$\bullet$  \ref{imifreq}};
\node at (0,-1.1)        {$\bullet$  \ref{finite}};
\node at (0,-1.9)        {$\bullet$  \ref{finite}};
\node at (0,3.2)         {$\bullet$  \ref{infinite}};

\end{tikzpicture}

\vspace{3mm}
\caption{Summary of the game of mass
$(\bf{m},\mathbb{X})$. The above rectangle offers a
visual classification of the possible different
masses. The region in gray corresponds to masses for
which the LLN is valid, that is,  
$S_t$ converges.
The vertical line divides the masses between regular
(left) and irregular (right) ones according to
definition \ref{emp_regular}.
The horizontal line separates the mass sequences
between bounded (down) and unbounded ones for which
$\limsup m_k =\infty$ (up).
Among the unbounded masses, those divergent in a Cèsaro
sense, and in particular those divergent in a classical
sense, are always regular.  
The black and red dotted boxes correspond to those
masses for which the related frequencies are
asymptotically stable, respectively, in a weak and in
a $L^1$ sense. The roman numbers in each of the
different subclasses correspond to the labels of the
different illustrative examples from
Sections~\ref{Exboun}-\ref{Exunboun}-\ref{Exran}.
} 
\label{fig:CRE}
\end{flushleft}
\end{figure}
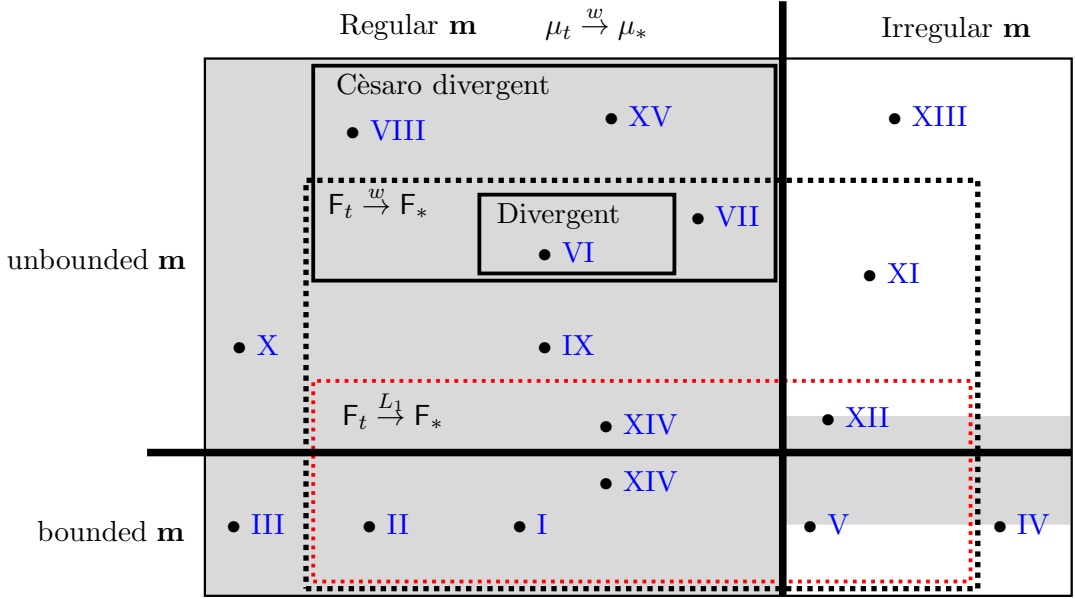     
\subsection{Regular mass sequences} \label{regular}  
Let $\Pc$ be the space of Borel measures on $\Rc$.
Recall~\eqref{lastresampling}, for a given weight
sequence  $\mathbf{m} \in \mathbb{R}_+^\mathbb{N}$,
define $\chv{\mu_t(\cdot) = \mu_{t}^{(\mathbf{m})}(\cdot), t \geq 0}$ to be the sequence of
\emph{empirical mass measures} on $\Rc$, where
$\mu_t(\cdot)$ is given by
\begin{equation}\label{empirical}
\mu_t(\cdot) 
: = \frac{\bar{t}}{t} \delta_{\bar{t}}(\cdot) 
+ \sum_{k = 1}^{\ell_t-1} \frac{m_k}{t} \delta_{m_k}(\cdot).
\end{equation}
Given a measure $\lambda \in \Pc$ and a measurable
function $f: \Rc \to \mathbb{R}$,
we denote by $\pd{\lambda}{f} = \int f(m)d\lambda(m)$
the integral of $f$ with respect to $\lambda$.
Furthermore, we say that a sequence
$\chv{\lambda_t;t\geq 0}$
of probability measures on $\Rc$ converges \emph{weakly}
to a probability measure $\lambda_*$ on $\Rc$,
denoted by $\lambda_t \xrightarrow[]{w}\lambda_*$ when
\begin{equation}\label{emp_C0}
\lim_t \pd{\lambda_t}{f} =  \pd{\lambda_*}{f}, \quad \text{for every } f \in \CL(\Rc),
\end{equation}
where $\CL(\Rc)$ denotes the space of continuous
functions on $\Rc$.
Note that this definition allows for
$\lambda_*(\{\infty\}) := 1 - \lambda_*(\mathbb{R}_+)$
to be strictly positive. 
\begin{definition}[{\bf Regular mass sequence}]
\label{emp_regular}
We say that $\mathbf{m}$ is  a regular mass sequence
when there is a probability measure $\mu_* \in \Pc$
such that
\begin{equation}\label{empregular}
\mu_t\xrightarrow[]{w} \mu_*.
\end{equation}
\end{definition}
The following proposition determines the limit of
$\mathcal{S}_t$ for regular mass sequences.
\begin{proposition}[{\bf Limit characterization for regular sequences}]
\label{cararegular}
{\, }\\
Assume $\mathbb{X}$ satisfies~\eqref{S1}--\eqref{S3},
and also that it satisfies~\eqref{samemean}
and~\eqref{vinfty}.
Then, for any mass sequence $\mathbf{m}$ and any
$t\geq 0$:
\begin{equation}\label{Smean}
  \mathbb{E}[\mathcal{S}_t]= \int v_m d\mu_t(m).
\end{equation} 
In particular, if $\mathbf{m}$ is regular
and~\eqref{empregular} holds true, then 
\begin{equation}\label{regularSt}
\mathbb{P}\prt{\lim_t \mathcal{S}_t = \int v_m d\mu_*(m)} = 1.
\end{equation}
\end{proposition}

\begin{proof}
To prove \eqref{Smean} note first that by \eqref{vinfty}
we can define $f \in \CL(\Rc)$
by $f(\infty): = \lim_m v_\infty$ and $f(m) = v_m$.
Now note that
\begin{align*}
  \mathbb{E}[\mathcal{S}_t] &=
  \mathbb{E}\Big[\frac{\bar{t}}{t} X_{\ell_t}(\bar{t})
    + \sum_{k = 1}^{\ell_t-1}\frac{m_k}{t}X_{k}(m_k)\Big]\\
  & =  v_{\bar{t}} \frac{\bar{t}}{t} +
  \sum_{k = 1}^{\ell_t-1} v_{m_k} \frac{m_k}{t}
  = \int v_m d\mu_t(m)= \pd{\mu_t}{f}.
\end{align*}
As a consequence, if $\mathbf{m}$ is regular there is
$\mu_* \in \Pc$ for which~\eqref{empregular} holds. 
Therefore, by~\eqref{empregular}
$\pd{\mu_t}{f} \to \pd{\mu_*}{f}$ and~\eqref{regularSt}
follows from Corollary~\ref{gameofmass}.
\end{proof}
\begin{remark}
When $\mathbf{m}$ is not regular, almost sure
convergence is not prevented, in fact, if $v_m = 0$ for
all $m$,
then by Theorem \ref{gradualstrong}, $\mathcal{S}_t$
converges almost surely to $0$. On the other hand,
Examples \ref{r:emp},~\ref{imifreq} presented in
Section \ref{cesarounbounded}
below show that almost sure convergence may not hold
for irregular masses.
\end{remark}
\subsubsection{Regularity and stability of empirical frequency}\label{frequency}
There are other notions of regularity beyond the one
captured in Definition \ref{emp_regular}.
For example, instead of
the empirical  measure in \eqref{empirical},
we may examine 
the \emph{empirical mass frequency}
$\chv{\mathsf{F}_t = \mathsf{F}_{t}^{(\mathbf{m})}, t \geq 0}$,
where $\mathsf{F}_t\in \Pc$ is given by
\begin{equation}\label{e:empfreq}
\mathsf{F}_t 
:= \frac{\delta_{\bar{t}}}{\ell_t}
+\sum_{k=1}^{\ell_t-1}\frac{ \delta_{m_k}}{\ell_t}.
\end{equation}
The reason to consider other notions of regularity
is that it contributes to a finer control of the
operative conditions for convergence for $L^1$ bounded
sequences of increments,
see also exemple \ref{finite}.

We note that, for any $t\geq 0$ and any arbitrary
function $f$,
the following relation between $\mu_t$
and $\mathsf{F}_t$ is in force: 
\begin{equation}\label{FandMu}
\int f(\x) \, d\mu_t(\x)=\frac{\ell_t}{t}\int \x f(\x) \, d\mathsf{F}_t(\x) .
\end{equation}
In particular, if we take $f(m)=v_m$ and $f(m)\equiv1$,
respectively, we have that
\begin{equation}\label{e:empmv}
  \mathbb{E}[\mathcal{S}_t]
  =\frac{\ell_t}{t}\int v_m m \, d\mathsf{F}_t(m),
\end{equation}
and 
\begin{equation}\label{lt}
 \frac{t}{\ell_t}= \int m \, d \mathsf{F}_t(m).
\end{equation}
The relation in \eqref{FandMu} may suggest to consider
weak convergence of
$\mathsf{F}_t$ as a natural alternative notion of
regularity.
However, as shown in the Proposition \ref{cararegular}
below, these two notions are not equivalent.
We find more convenient to adopt the notion in
Definition \ref{emp_regular} for the following two
reasons.
First, there are masses for which both $\mu_t$
and $\mathsf{F}_t$
converge weakly to some $\mu_*$ and $\mathsf{F}_*$,
respectively,
but the limit of $\mathcal{S}_t$ is determined by $\mu_*$
and not by $\mathsf{F}_*$, see Examples \ref{F0},
\ref{triangular},
and~\ref{cesar} below.
Second, among the unbounded masses,
those divergent in a Cèsaro sense will always be regular
while the corresponding $\mathsf{F}_t$ is not
guaranteed to admit a limit,
see Examples \ref{cesaroifreq} and \ref{unbnofreq}.

Yet, it is interesting to look at the LLN from the
perspective of masses with ``well-behaved'' frequencies.
In particular, the next proposition clarifies how the
relation between $\mu_t$ and $\mathsf{F}_t$ expressed
in \eqref{FandMu} behaves in the limit. In particular
it shows how to relate the behavior of the empirical
frequencies and empirical masses under different modes
of convergence, which we next define.

As is standard, see \cite[Thm. 4.6.3, p. 245]{Dur19} we
write $\mathsf{F}_t\xrightarrow[]{L^1}\mathsf{F}_*$, if
there exists a measure $\mathsf{F}_*(\cdot)$ on
$\mathbb{R}_+$ for which
\begin{equation}\label{L1discrete}
  \mathsf{F}_t \xrightarrow[]{w}\mathsf{F}_* \quad \text{ and }\quad  \int m \, d\mathsf{F}_t(m) \to \int m \, d\mathsf{F}_*(m)< \infty.
\end{equation}
In a somewhat dual manner,
we write ${\mu}_t\xrightarrow[]{w^+}\mu_*$, if 
\begin{equation}\label{minus1}
  \mu_t \xrightarrow[]{w}\mu_* \quad \text{ and }\quad  \int \frac{1}{m} \, d\mu_t(m) \to \int \frac{1}{m} \, d\mu_*(m)<\infty.
\end{equation} 
\begin{proposition}[{\bf Regularity and stable frequencies}]
\label{stableregularity}
Assume $\mathbb{X}$
satisfies~\eqref{samemean} and~\eqref{vinfty},
consider a mass sequence $\mathbf{m}$ and assume that
for $\ell_t = \ell_t(\mathbf{m})$, the limit
$A:=\lim_{t\to\infty}\frac{\ell_t}{t}\in \Rc$ exists.
Then:
\begin{description} 
\item[\namedlabel{aa}{a}]
  $\mathsf{F}_t \xrightarrow[]{L^1}\mathsf{F}_* \neq \delta_0
  \Rightarrow \mathsf{\mu}_t \xrightarrow[]{w}\mathsf{\mu}_*$
  with
  $\pd{\mu_*}{f} := A \int m f(m)\, d \mathsf{F}_*(m)$,
\item[\namedlabel{mathbb}{b}]
$ \mathsf{\mu}_t \xrightarrow[]{w^+}\mathsf{\mu}_*\neq \delta_\infty
\Rightarrow  \mathsf{F}_t \xrightarrow[]{w}\mathsf{F}_*$
with
$\pd{\mathsf{F}_*}{f} :=  \frac{1}{A} \int \frac{1}{m}f(m)\, d \mu_*(m)$.
\end{description}
Furthermore if both cases above $A \in (0,\infty)$.
\end{proposition}

\begin{proof}
For item \ref{aa}, by \eqref{lt}, \eqref{L1discrete}
and the assumption that $\mathsf{F}_t(\cdot)$ converges
in $L^1$ to $\mathsf{F}_* \neq \delta_0$, we have that
\begin{equation}\notag
\frac{t}{\ell_t} =   \int m dF_t(m) \to A = \int m d F_*(m) \in(0, \infty) .
\end{equation}
The above relation, \eqref{FandMu},
and \eqref{L1discrete} 
imply that for any $f \in \CL(\Rc)$
\begin{equation}\label{Fmuregular}
\mu_t(f) = \frac{\ell_t}{t}\int m f(m)\, d \mathsf{F}_t(m) \to A  \int m f(m)\, d \mathsf{F}_*(m).
\end{equation}
We now turn to the proof of item \ref{mathbb}.
From \eqref{FandMu}, the convergence of $\pd{\mu_t}{f}$
applied to the function $f(m)=1/m$
and the assumption that
$\mathsf{\mu}_*\neq \delta_\infty$ we have that
\begin{equation}\label{AAA}
  \frac{\ell_t}{t} = \int \frac{1}{m} \,d\mu_t(m)
  \to \int \frac{1}{m} \, d\mu_*(m) = A \in (0,\infty).
\end{equation}
Therefore, for any $f\in \CL(\Rc)$, by~\eqref{FandMu},
\eqref{minus1}, and~\eqref{AAA} we conclude that 
\begin{equation}
  \pd{ \mathsf{F}_t}{f} = \frac{t}{\ell_t}
  \int \frac{1}{m}f(m)\, d \mu_t(m)
  \to \frac{1}{A} \int \frac{1}{m}f(m)\, d \mu_*(m).
\end{equation}
\end{proof}

Proposition~\ref{stableregularity} explains part of the
different relations depicted in Fig.~\ref{fig:CRE}
among the dotted boxes corresponding to masses for
which $F_t$ convergences weakly and in $L^1$.
In what follows, with the help of examples,
we explore more how these notions of weak
and $L^1$ convergence for $F_t$
relate to the regularity of $\mu_t$.
The examples are organized in the following sections,
and in particular they clarify
Figure~\ref{fig:CRE}.
First, in Section~\ref{Exboun},
we examine bounded masses.
Second, in Section~\ref{cesarounbounded},
we examine Cèsaro divergent masses.
Third, in  Section~\ref{notcesaro},
we consider unbounded masses that
are not Cèsaro divergent.
To conclude, in Section~\ref{Exran},
we examine  i.i.d.  random masses.
\subsection{Examples of bounded masses}\label{Exboun} 
If the sequence of $\mathbf{m}$ is bounded then weak
convergence of $\mathsf{F}_t$ implies $L^1$ convergence
of $\mathsf{F}_t$.
We also remark that when the sequence is regular the
limit of $\mathcal{S}_t$ exists and is given by
$v = \int v_m \, d\mu_*(m)$.
The following examples show how
regular masses relate
with weak convergent frequencies.

\begin{example}[\textbf{Regular} + $L^1$-$\lim \mathsf{F}_t$]\label{a}\rm{}
When $\sup_k m_k <\infty$, $L^1$ convergence
follows from weak convergence
of the empirical frequency plus uniform integrability.
If $\mathsf{F}_*(m) \neq \delta_0$,
then the formula for the limit of $\mathcal{S}_t$ is
given in terms of $\mathsf{F}_*$.
Indeed, by item~\ref{aa}
of Proposition~\ref{stableregularity}
we conclude that
\begin{equation}\label{e:snc} 
\mathbb{P}\prt{\lim_{t \to \infty}\mathcal{S}_t = v }=1
\quad \text{where} \quad
v = A \int v_m m  \, d\mathsf{F}_*(m), \quad A: = \lim_t \frac{\ell_t}{t}.
\end{equation}
\end{example}
\begin{example}[\textbf{Regular} + $L^1$-$\lim \mathsf{F}_t$]\label{F0}\rm{}
This example shows that if the limit
$\mathsf{F}_* = \delta_0$ then $\mu_*$ may not be
given by the expression in item \eqref{aa} of
Proposition~\ref{stableregularity}.
  
Consider the triangular array
$\chv{{a}_{i,j}, i,j\in \mathbb{N}, j\leq i}$
defined by  $a_{i,1} := 1$ and for
$1<j\leq i$, $a_{i,j} := 2^{-i}$, represented below
\begin{equation}\label{seqF0}
  \begin{aligned}
  &1,\\
  &1,2^{-1},\\
  &1,4^{-1},4^{-1},\\
  &1,8^{-1}, 8^{-1}, 8^{-1}\\
  &\ldots\\
\end{aligned}
\end{equation}
For the sequence of increment take $m_k$ to be
the $k$-th term of this array, more precisely let
$i(k)$ be such that
\[
\frac{(i(k)-1)i(k)}{2} \leq k \leq \frac{i(k) (i(k) + 1)}{2}
\]
and $j(k) := k - \frac{(i(k)-1)i(k)}{2}$.
Let $m_k := a_{i(k),j(k)}$.
In this example,
$\mathsf{F}_t \overset{L^1}{\to}\delta_0$ while
$\mu_t \overset{w}{\to} \delta_1$. This shows that the
$L^1$ limit of $F_t$ is not sufficient to describe the
limit of $\mathcal{S}_t$, which is given by
$v_1 = \int v_m \, d \mu_*(m)$.
\end{example}

\begin{example}[\textbf{Regular} +  $\nexists\,\lim \mathsf{F}_t$]\label{Fnot}\rm{}
Let us now move to an  example of bounded regular
mass sequence such that the limit of $\mathsf{F}_t$
does not exist.
Consider the sequence $\mathbf{m}$ defined as follows:
\begin{itemize}
\item[(i)]  Set  $m_1 = 1$,
\item[(ii)] while  $\mathsf{F}_{M(k)}(\{1\})> 1/4$ set
  $m_k = 2^{-k}$  else, go to (iii),
\item[(iii)] while   $\mathsf{F}_{M(k)}(\{1\})< 3/4$
  set $m_k = 1$ else, go to (ii).
\end{itemize}
In this case, $\mu_t \overset{w}{\to} \delta_1$ and
$\mathsf{F}_t$ does not converge.
\end{example}
Note that if $\mathbf{m}$ is not regular, then
depending on $\chv{v_m, m\in \mathbb{R}_+}$,
$\mathbb{E}[\mathcal{S}_t]$  may or may not converge.
If there are $K,L \in \mathbb{R}_+$ such that
$v_K < v_L$, as in the example below, it is simple to
construct a sequence $\mathbf{m}$ for which
$\mathbb{E}[\mathcal{S}_t]$ does not converge.
\begin{example}[\textbf{Irregular} + $\nexists\, \lim \mathsf{F}_t$]
\label{irr}\rm{}
Let $\mathbf{m}$  be the sequence composed of $A_i$
increments of size $K$ followed by $B_i$ increments of
size $L$ where the sequences $(A_i)_i, (B_i)_i$ will be
determined later. More formally, let $(A_i, B_i)_i$ be
given, define
$\tau_0 := 0$ $\tau_n := \tau_{n-1} + A_n + B_n$ and set 
\begin{equation}\label{e:seq3}
m_k =
\begin{cases}
  K &\text{ if $ k \in (\tau_n, \tau_n + A_{n+1}]$ \;\,\;\;for some $n\geq 0$,}\\
  L &\text{ if $ k \in (\tau_n+A_{n+1}, \tau_{n+1}]$ for some $n \geq 0$}.
\end{cases}
\end{equation}
Choose $(A_i,B_i\,;\,i \in \mathbb{N})$ such that for
all $n \in \mathbb{N}$, $A_n < A_{n+1}$, $B_n<B_{n+1}$ and
\begin{equation}\label{irregkl}%
  \frac{L(B_1 + \ldots + B_{n})}{K(A_1 + \ldots + A_{n+1})}
  \leq \tfrac1n
  \quad\text{ and }\quad
  \frac{K(A_1 + \ldots + A_{n+1})}{L(B_1 + \ldots + B_{n+1})}
  \leq \tfrac1n.
\end{equation}
If $v_K < v_L$ then $\mathbb{E}[\mathcal{S}_t]$ does
not converge as
\begin{equation}\label{}%
\limsup_t \mathbb{E}[\mathcal{S}_t]= v_L \neq v_K=\liminf_t \mathbb{E}[\mathcal{S}_t].
\end{equation}
\end{example}
\begin{example} [\textbf{Irregular +} $L^1$-$\lim\mathsf{F}_t$] 
\label{irregbounded} \rm{}
If we combine the sequence defined in Example \ref{F0}
with the one defined in Example  \ref{irr} we can
construct an irregular sequence for which
$F_t \overset{w}{\to}F_*$. More precisely, let $m'_{k}$
be the sequence defined in Example \ref{irr} and
consider a triangular array $a_{i,j}$ defined by
$a_{i,1} := m'_i$ and  for $1<j\leq i$, set
$a_{i,j} := 2^{-i}$. To conclude, set
$m_k := a_{i(k),j(k)}$ with $i(k), j(k)$ as defined in
Example \ref{F0}. Note that this sequence is irregular
even though $\mathsf{F}_t\overset{L^1}{\to}\delta_0 $.
\end{example} 

As we look back to item \eqref{mathbb} of
Proposition~\ref{stableregularity} we see that $w^+$
convergence can not occur in any of the examples of
bounded regular mass for which  the empirical frequency
does not converge. Indeed, all those example have a
significant amount of increments of negligible mass,
and as such, they modify
the empirical frequency without affecting the limit of
the mass sequence.
We now move to the study of unbounded masses. 
\subsection{Unbounded masses} \label{Exunboun}
\subsubsection{Divergent and Cèsaro's divergent masses} \label{cesarounbounded}
We say that a  sequence of masses $\mathbf{m}$ is
Cèsaro's divergent when 
\begin{equation}\label{cesarodef}%
\lim_n \frac{m_1 + \ldots + m_n}{n} \to \infty
\end{equation}
In this case one has that
$\mu_t \overset{w}{\to} \delta_\infty$.
Therefore the  Cèsaro divergent sequences  are always
regular
and by Proposition~\ref{cararegular} 
it follows that
\begin{equation}\label{cesarospeed}
\mathbb{P}( \lim_t \mathcal{S}_t = v_\infty) =1.
\end{equation} 

A particular case of Cèsaro divergence is given by the
divergent masses as captured in the next example.
\begin{example}[\textbf{Divergent mass} + $w$-$\lim \mathsf{F}_t$]\label{mdivergent}\rm{}
We say that a sequence of masses $\mathbf{m}$ is
divergent when
\begin{equation}\label{diverging}%
\lim_{k \to \infty} m_k = \infty,
\end{equation}
in which case \eqref{cesarodef} holds true and hence
\eqref{cesarospeed}. Also, as
$\lim_t \mathsf{F}_t([0,A]) = 0$ for any $A>0$,
it follows that
$\mathsf{F}_t \overset{w}{\to} \delta_\infty$.
\end{example}     
The regime captured in~\eqref{diverging}
is treated in~\cite[Theorem 1.10]{ACdCdH19}
in the context of random walks in dynamic environment. 
Theorem~\ref{gradualstrong} can actually
be seen as a generalization of the latter. 
As briefly mentioned at the end
of Section~\ref{hypothesis},
the present proofs could actually
cover even more general cases, if,
for example, we relax the assumption
in Equation~\eqref{samemean}.
The following example shows that in
the Cèsaro divergent regime,
the sequence $\mathsf{F}_t$ may converge,
but may not be able to capture the limit of $\mathcal{S}_t$.
\begin{example}
[\textbf{Cèsaro divergent mass} + $w$-$\lim \mathsf{F}_t$]
\rm{}\label{cesar}
Consider the sequence $\mathbf{m}$
\begin{equation}\label{e:seq1}
m_k :=
\begin{cases}
  1 &\text{ if $ k $ is odd, and }\\
  k &\text{ if $ k $ is even}.
\end{cases}
\end{equation}
Informaly,  half the increments are $1$, and the other
half diverges. More precisely,
\begin{equation}\label{half1infty}
\mathsf{F}_t \overset{w}{\to}\frac{1}{2}\delta_1 +\frac{1}{2}\delta_\infty .
\end{equation}
As such, one might be tempted to say that 
$\mathbb{E}\crt{\mathcal{S}_t} \to \tfrac12 v_1 + \tfrac12 v_\infty$ as $t \to \infty$.
This is not the case because one has to take into
account the relative weights of the sequences.
As it turns out,  the mass of increments of size $1$
for this particular sequence vanishes in the limit.
Indeed, note that the sum of the first $2k$
increments, $M_{2k}$ is
\[ 
M_{2k}= \frac{k(k-1)}{2} + k = \frac{k^2 + k}{2}.
\]
Now note that $\frac{k}{M_{2k}} \to 0$ and therefore 
\begin{equation}\label{}%
\mathbb{E}[\mathcal{S}_{M_{2k}}] =\frac{k}{M_{2k}}v_1 + \frac{1}{M_{2k}} \sum_{i= 1}^k i v_i \to v_\infty.
\end{equation}
Also in this example, if $v_1\neq v_\infty$, then the
weak limit of $\mathsf{F}_t$ does not determine the
limit of $\mathcal{S}_t$,
even if it is well defined.
\end{example}
As in the bounded case, see Example \ref{Fnot}, also
Cèsaro divergent sequences may not have well behaved
empirical frequencies,
as shown in the next example.
\begin{example}[\textbf{Cèsaro divergent mass} + $\nexists\, w$-$\lim \mathsf{F}_t$]\label{cesaroifreq}\rm{} 
It is possible to construct a sequence $\mathbf{m}$
that is regular but such that $\mathsf{F}_t$ does not
converge weakly. Take an irregular sequence
$\mathbf{m}'$ such as the one defined in \eqref{e:seq3}
and intercalate it with a huge increment so that it
diverges in the Cèsaro sense. To be more concrete, for
$k \in \mathbb{N}$ let $m_{2k -1} : m'_k$ and
$m_{2k}: =k\sum_{i = 1 }^{2k-1} m_{i}$.
\end{example}

\subsubsection{Unbounded sequences that do not diverge in the Cèsaro sense}
\label{notcesaro}
When dealing with unbounded masses that
are not Cèsaro divergent,
then the sequence is not necessarily
regular and more subtle scenarios may occur,
as the following examples illustrate.
We start with an example of a regular sequence that
allows an asymptotic positive mass of increments of finite size and positive mass at infinity.

\begin{example}[\textbf{Regular} $\liminf m_k<\infty$ + $w$-$\lim \mathsf{F}_t$]\label{triangular}\rm{}
Let  the first elements of
$\mathbf{m}\in\mathbb{R}_+^{\mathbb{N}}$ be defined by
$m_1:=1, m_2:=2$ $m_ 3:= 1$.
If $m_k := j > 1$  then the next $j-1$ increments will
be of size $1$ after that $m_{k+j} := j+1$.
The sequence of increment sizes can be arranged in a
triangular array $\chv{a_{i,j}}_{i,j \geq 1}$, where
$m_k := a_{i(k),j(k)}$ with $i(k), j(k)$ as in
Example~\ref{F0}.
\begin{equation}\label{e:seq2}
  \begin{aligned}
  &1,\\
  &2,1,\\
  &3,1,1,\\
  &\ldots
  \end{aligned}
\end{equation}
In this case
$\mathsf{F}_t \overset{w}{\to} \delta_1$ but
$\mu_t \overset{w}{\to} \tfrac12 \delta_1 + \tfrac12 \delta_\infty $
and so
$\lim_t \mathbb{E}[\mathcal{S}_t] = \tfrac12 v_1 + \tfrac12 v_\infty$.
We notice in particular that if $v_1>v_\infty$, the
above mass sequence is another example
of a regular sequence  for which
the weak limit of $\mathsf{F}_t$ does not determine the
limit of $\mathcal{S}_t$, 
even when it exists.
\end{example}
The next example shows a regular sequence with
unbounded increments and for which the empirical frequency does not converge.
\begin{example}[\textbf{Regular} + $\nexists$ $w$-$\lim \mathsf{F}_t$]\label{unbnofreq}\rm{}
Take $\mathbf{m}$ as in Example \ref{Fnot} but
replace the $k$-th increment of mass $1$  by the
$k$-th increment  of the sequence defined in
Example~\ref{triangular}. For this example, we have that
\[
\limsup_t\mathsf{F}_t(\{1\}) - \liminf_t \mathsf{F}_t(\{1\}) \geq 1/2.
\]
The above condition precludes the weak convergence of
$\mathsf{F}_t$.
Furthermore, since the total mass on increments
strictly smaller than $1$ is finite,
$\mu_t \overset{w}{\to} \frac{1}{2}\delta_1 +\frac{1}{2}\delta_\infty$
and the sequence is regular.
\end{example}

\begin{example}
[\textbf{Irregular} + $w$-$\lim\mathsf{F}_t$] 
\label{r:emp} \rm{}
Only weak convergence of the empirical measure
$\mathsf{F}_t$  does not imply convergence
of $\mathcal{S}_t$. 
Indeed, assume that $v_1 >v_\infty$, and let $(K_i)_i$
and $(N_i)_i$ be auxiliary sequences that we will
determine later.
The sequence $\mathbf{m}$ alternates one increment of
size $K_i$ with $N_i$ increments of size $1$.
More precisely, let $\tau(j) = j + \sum_{i = 1}^j N_i$
and set
\[
m_k = \sum_{j=1}^\infty K_j \Ind{\tau(j)} + \Ind{(\tau(j), \tau(j+1))}.
\]
Now, choose $(N_i,K_i)$ such that
\begin{equation}\label{}%
\frac{N_1 +\ldots +N_i}{K_i} \leq \tfrac1i \text{ and } \frac{K_1 + \ldots + K_i}{ N_{i+1}} \leq \tfrac1i.
\end{equation}
Note that  $\mathsf{F}_t \overset{w}{\to} \delta_1$,
but if $v_1 >v_\infty$
\begin{equation}\label{}%
  \limsup_t \mathbb{E}[\mathcal{S}_t] = v_1 > v_\infty
  = \liminf_t \mathbb{E}[\mathcal{S}_t].
\end{equation}
\end{example}
\begin{example}[\textbf{Irregular }+  $L^1$-$\lim \mathsf{F}_t$]\label{irregL1unbounded}\rm{}
In this example we construct an unbounded irregular
sequence for which $\mathsf{F}_t$ converges in $L^1$.
In particular from item \eqref{aa} of
Proposition~\ref{stableregularity} it follows that
this limit must be $\delta_0$.
Let $(A_i)_i$ be auxiliary sequences to be defined
later. Informally,  this constructed as a combination
of Example~\ref{irregbounded} and Example~\ref{r:emp},
where we intercalate an irregular unbobunded sequence
with a large number of increments os small mass. Let
$\mathbf{m}'$  be the sequence defined in
Example~\ref{r:emp}  set  $\tau(1) := 1 $ and  for $j >1$ set $\tau(j) := \tau(j-1) + A_j$. Now let
\begin{equation}\label{interL1irr}
m_k : = \sum_j m'_j \Ind{\tau(j)} + 2^{-k} \Ind{(\tau(j-1), \tau(j))}.
\end{equation}
Finally choose $A_i$ such that 
\[
\frac{\sum_{k= 1}^{i}m'_k +1}{A_i} < \frac{1}{i}.  
\]
Since
\[
\int m dF_t(m)  \leq \frac{\sum_{k = 1}^{\ell_t} m'_k +1}{\ell_t} \to 0, \quad \text{as } t \to \infty,
\]
it follows that
$\mathsf{F}_t\overset{L^1}{\to}\delta_0$.
Furthermore, if $v_1 > v_\infty$
\[
\limsup_t \mathbb{E}[\mathcal{S}_t]= v_1 > v_\infty=\liminf_t \mathbb{E}[\mathcal{S}_t].
\]
\end{example}
 
\begin{example}[\textbf{irregular} + $\nexists\,$ $w$-$\lim \mathsf{F}_t$ ]\label{imifreq}\rm{}
It is also possible to construct a sequence $\mathbf{m}$
that is irregular but such that $\mathsf{F}_t$ does
not converge weakly.
Take the sequence defined in Example \ref{Fnot} and
replace the $k$-th increment of size $1$ by the $k$-th
increment of the sequence  defined in \ref{r:emp}.
\end{example}

\subsection{Random masses} \label{Exran}
In this section 
consider random  mass sequences $\mathbf{m}$.
More specifically, we let $m_k$ be an i.i.d. sequence
of random variables, independent of $\mathbb{X}$, each
distributed according to a measure $\nu$ on
$\mathbb{R}_+$. There are two cases  depending on
weather $\nu$ has finite or infinite mean.          
For notational ease, we model
$\prt{m_k, k \in \mathbb{N}}$ as i.i.d. random
variables in the probability space
$\prt{\Omega, \mathcal{F}, \mathbb{P}}$. 
\begin{example}[\textbf{Regular} + \textbf{(un) bounded }+ $L^1$-$\lim\mathsf{F}_t$]\label{finite}\rm{}
Assume that $\nu \neq 0$, and that
$\int m\nu(dm)<\infty$.
Now, let the increments $m_k$ to be sampled
independently from $\nu$. 
By the Glivenko-Cantelli Theorem
\cite[Theorem 2.4.9]{Dur19} it follows that almost
surely $\mathsf{F}_t([0,x])$ converges
(uniformly in $x$) to $\nu([0,x])$.
By the classical LLN for iid. random variables, almost
surely, $\int m d\mathsf{F}_t(m) \to \int m d\mathsf{F}_*(m)$.
Therefore the conditions of~\eqref{L1discrete} are
satisfied almost surely and so
$\mathbb{P}(\mathsf{F}_t \overset{L^1}{\to} \mathsf{F}_*) = 1$.
By item~\eqref{aa} of
Proposition~\ref{stableregularity} it follows that
$\mathbb{P}\big(\mu_t \xrightarrow[]{w} \nu\big) = 1$.
Therefore, almost surely, the sequence $m_k$  is
regular and
\begin{equation}
\mathbb{P}\Big(\lim_t \mathbb{E}[\mathcal{S}_t] = \int v_x d\nu(x) \Big)=1.
\end{equation}
\end{example}

\begin{example}[\textbf{Regular} + \textbf{Cèsaro }+ $\exists \,w$-$\lim F_t$ ]\label{infinite}\rm{}
Now, assume that $\int m d\nu(m) = \infty$ and define
the increments $m_k$ to be sampled independently
from $\nu$. 
In this case
\begin{equation}\label{diverge}%
  \mathbb{P}\prt{ \frac{m_1 +  \ldots + m_k }{k}
    \to \infty}=1.
\end{equation}
Then note that after $k$ increments, the mass of
increments of size smaller than $a>0$, $\mu_t([0,a])$,
is bounded by $\frac{ka}{m_1 + \ldots + m_k}$ and
therefore, by \eqref{diverge}, for any $a>0$, almost
surely
$\mu_t([0,a]) \to 0.$
This implies that
$\mathbb{P}(\mu_t \overset{w}{\to} \delta_{\infty})=1$
and therefore
\begin{equation}\label{}%
\mathbb{P}\prt{\lim_t \mathcal{S}_t = v_\infty} = 1.
\end{equation}
\end{example}

\section[]{Hypothesis and counterexamples}\label{hypothesis}
\subsection{Weak LLN: necessity of (W1) and (W2).} 
Booth conditions (W1) (W2) are necessary for the weak
LLN.
The necessity for condition (W1) is due
to~\cite[Theorem 1]{JamOrePru65}.
We show below that condition (W2) is necessary by means
of a counter-example.

\paragraph{Counter-example:}
Consider a sequence $\chv{U_k, k \in \mathbb{N}}$ of
i.i.d. uniform random variables on $(0,1)$ and
$X_k(m) := V_m(U_k)$, where
\begin{equation}
V_m(u) = 
\begin{cases}
A_m & \text{ if } u \in [0,g(m)/2),\\
-A_m& \text{ if } u \in (g(m)/2, g(m)],\\
0   & \text {else.}
\end{cases}
\end{equation}
with this definition, it follows that
$\mathbb{P}\prt{\abs{X_k(m_k)}>0} = g(m_k)$.
Assume that $g$ is a strictly decreasing  continuous
function such that $\lim_{m \to \infty} g(m)=0$.
Let $m_k: = \inf\chv{m: g(m)>1/k}$. 
This implies that $m_k \to \infty$ as $k \to \infty$
and so \eqref{divergent} is satisfied.
Furthermore by the definition of $X_k(m_k)$, the
assumptions~\eqref{C} and~\eqref{W1} in
Theorem~\ref{incrementalweak} are verified.
Now choose $\chv{A_{m_k}, k \in \mathbb{N}}$ to be such
that
\begin{equation}
  \frac{m_n}{M_{N(n)}} A_{m_n} > 1 +
  \sum_{k = 1}^{n-1}  A_{m_k},
\end{equation}
where  $N(n)$ is such that
\begin{equation}\label{futureprotected}
\mathbb{P}(\exists\, {n\leq j \leq N(n)}: X_j(m_j) \neq 0)>\frac{1}{2}.
\end{equation}
Such an $N(n)$ exists and is finite since by the second
Borel-Cantelli Lemma and the continuity of probability
measures:
\begin{equation}
1 = \mathbb{P}(\exists \, j\geq n : X_j(m_j) \neq 0) = \lim_{N \to \infty}\mathbb{P}(\exists \,n\leq j <N : X_j(m_j) \neq 0).
\end{equation}
With this choice of $A_{m_n}$ it follows that if there
is a $j$, $i \leq j \leq N(i)$ for which
$\abs{X_j(m_j)}>0$ then $\abs{S_{N(i)}}>1$.
Therefore  for any $i \in \mathbb{N}$,
\begin{equation}\label{counter}
\mathbb{P}\prt{\abs{S_{N(i)}}> 1} > \frac{1}{2}.
\end{equation}
As $\mathbb{P}(S_{n}>0 \mid \abs{S_n}>0) = \frac{1}{2}$, we conclude that
the weak LLN does not hold.

\subsection{Strong LLN: near optimality of (S1).}
One could try to improve the condition in \eqref{S1}
by requiring a decay smaller than polynomial, that is: 
\begin{equation}\label{fconcentration}
\mathbb{P}\left(\abs{X_k(m)}>\gep\right)<\frac{C_\gep}{f(m)},
\end{equation}
for some  $f: \mathbb{R}_+ \to \mathbb{R}_+$.
When we look for a scale that grows slower than any
polynomial, $f(m) = \log(m)$ is a natural candidate.
However, as illustrated next, this already allows for
counterexamples. 
\paragraph{Counter-example:} 
Let $\chv{U_k, k \in \mathbb{N}}$ be a sequence of
i.i.d. uniform random variables on $(0,1)$ and let
$X_{k}(m) := g_m(U_k)$ where
\begin{equation}\label{e:nce}
g_m(x) :=
\begin{cases}
 \phantom{-}1  & \text{if }   x \in (0,\frac{1}{2\log_2 m}),\\
 -1 & \text{if }   x \in [\frac{1}{2\log_2 m},\frac{1}{\log_2 m}),\\
  \phantom{-}0 & \text{else}.\\
\end{cases}
\end{equation}
Note that $\mathbb{X}$ fulfills
assumptions~\eqref{C}, \eqref{S2}, \eqref{S3},
and instead of~\eqref{S1} it satisfies 
\begin{equation}\label{e:ncr}
\mathbb{P}\prt{\abs{X_{k}(m)}>\epsilon} = \frac{1}{\log_2m}.
\end{equation}
Now take $\bf{m}$ with $m_k = 4^k$. For such an
$\bf{m}$ we see that the incremental sum $S_n$ does not
satisfy the strong LLN. Indeed, as
\begin{equation}\label{e:bc2}
\mathbb{P}\prt{\abs{X_{k}(m_k)}=1} = \frac{1}{2k},
\end{equation} 
by the second Borel Cantelli lemma,
\begin{equation}\label{e:bc3}
\mathbb{P}\prt{\abs{X_k(m_k)}=1, \text{ i.o} } = 1.
\end{equation}
Therefore, by \eqref{Sn} it follows that there is an  $\gep > 0 $ for which
\[
\mathbb{P}(\abs{S_n - S_{n-1}} >\gep \text{ i.o.}) =1,
\]
which  means that almost surely $S_n$ does not converge.
\medskip

In light of the above example, we see that the
condition~\ref{S1} is near to optimal.
Indeed, to improve it, we would need to find $f(m)$
satisfying
\[
\log^k(m)<<f(m) << m^\delta\quad \forall \, k \in \mathbb{N}, \; \delta >0 .
\]

\subsection{Concluding remarks}
\begin{itemize}
\item
{\bf Independence.} Our examples above and proofs below
are based on the independence in $k$ of
$\chv{X_k(m), m \in \mathbb{R}_+, k \in \mathbb{N}}$. 
However, for certain choices of well-behaved mass
sequences $\bf{m}$,
it seems possible to adapt our arguments and still
obtain a strong LLN in presence of ``weak enough dependence'',
though the notion of ``weak enough dependence'' would
very much depend on the weight sequence and this is why
we did not pursue this line of investigation.
\item{\bf Relaxing condition~\eqref{samemean}.}
In the game of mass described in Section~\ref{game},
for simplicity, we have restricted our
analysis to variables with expected value independent
of $k$, as captured in assumption~\eqref{samemean}.
We note that this not really needed,
as we might, for example,
consider $X_k(m)$'s with expected value, say, $v_m$
and $v_m'\neq v_m$ depending on the parity of $k$.
Yet, the resulting analysis would branch into many
different regimes depending on how exactly
condition~\eqref{samemean} is violated. 
\item{\bf Fluctuations and large deviations.}
It is natural to consider
``higher order asymptotics'',
such as large deviations or
scaling limit characterizations,
for the sums  in \eqref{Sn} or \eqref{gradual}.
However, the analysis for this type of questions
relies heavily on the specific
distribution of the sequence of variables
$\mathbb{X}$ thus preventing
a general self-contained treatment.   
Still, it is interesting to note that
these other questions
can give rise to many subtleties and anomalous behavior.
This is well illustrated
by the specific RWCRE model in random media
introduced in~\cite{AdH19} that motivated the present
paper,
we refer the interested reader to  \cite{ACdCdH20}
for results on crossovers phenomena in related
fluctuations,
and to~\cite{ACdCdH19} for stability results of large
deviations rate functions.
\end{itemize}     
\section{Proofs}\label{proofs}
 
\subsection{Weak law of large numbers}

In this section we prove Theorem~\ref{incrementalweak}
by implementing a truncation
argument following the ideas from~\cite{JamOrePru65}.  
For each $K>0$, let $S^K_n$ represent the contribution
to $S_n$ coming from the increments larger than $K$, i.e.
\begin{equation}
S^K_n := \sum_{k = 1}^n \frac{m_k}{M_n} X_{k}(m_k) \Ind{m_k>K}.
\end{equation}
Now note that due to~\eqref{W1} and~\eqref{W2} it
follows that
\begin{equation}\label{Expcontrol}
\lim_{K \to \infty}\sup_{m>K}\sup_k\mathbb{E}\crt{\abs{X_k(m)}} =0.
\end{equation}
Indeed, for any $\gep>0$ and any $A>\gep$
\begin{equation}
\begin{aligned}
\mathbb{E}\crt{\abs{X_k(m)}}&\leq
\gep + A\, \mathbb{P}(\gep < \abs{X_k(m)} \leq A) +  \mathbb{E}\crt{\abs{X_k(m)} \Ind{\abs{X_k(m)}>A}}.
\end{aligned}
\end{equation}
the right hand side above can be bounded by $3 \gep$
using~\eqref{W1} and~\eqref{W2}
and since $\gep>0$ is arbitrary, \eqref{Expcontrol}
follows.
Now let $\bar{S}^K_n : = S_n - S^K_n$ be the
contribution to $S_n$ coming from the increments
smaller than $K$.
By the triangle inequality and the union bound it
follows that
\begin{equation}
  \mathbb{P} \prt{\abs{S_n}> \gep} \leq \mathbb{P}
  \prt{\abs{S^K_n} + \abs{\bar{S}^K_n}> \gep}
  \leq \mathbb{P} \prt{\abs{S^K_n} > \frac{\gep}{2}}
  + \mathbb{P} \prt{\abs{\bar{S}^K_n}>\frac{ \gep}{2}}.
\end{equation}
As $\sum_{k = 1}^n \frac{m_k}{M_n} \Ind{m_k>K} \leq 1$,
\eqref{Expcontrol} and Markov's inequality imply
\begin{equation}\label{remains0}
  \limsup_{K \to \infty}  \mathbb{P} \prt{S^K_n> \gep} =0,
\end{equation}
and therefore for any $K>0$,
\begin{equation}\label{remains}
\limsup_{n \to \infty}  \mathbb{P} \prt{\abs{S_n}> \gep} \leq \limsup_{n\to \infty}\mathbb{P} \prt{\abs{\bar{S}^K_n}>\frac{ \gep}{2}}.
\end{equation}
It remains to prove that the right-hand side above goes
to zero for arbitrary $\gep>0$.
To do so, consider the truncated random variables 
\begin{equation}\label{trunc}
Y_k(m_k):=X_k(m_k)\Ind{\abs{X_k(m_k)} < \frac{M_n}{m_k}}\Ind{m_k<K} 
\end{equation}
and notice that as $M_n \to \infty$,
\begin{equation}\label{maxto0}
\lim_n \max_{1 \leq k \leq n} \frac{m_k}{M_n} \Ind{m_k<K} = 0.
\end{equation}
Set $\bar{s}^K_n :=\sum_{k =1}^n Y_n(m_k)$. We will first
argue that this truncated sum $\bar{s}^K_n$
approximates well $\bar{S}^K_n$, and then show that the
variance of truncation vanishes. To perform these two
steps we will need the following lemma, whose proof is
postponed to the end of this section.

\begin{lemma}
\label{truncY}
If~\eqref{C}, \eqref{W1},
and~\eqref{W2}
of Theorem~\ref{incrementalweak} hold true, then
\begin{equation}\label{Jam651}
\lim_{n  \to \infty}  \max_{1 \leq k \leq n}\frac{M_n}{m_k} \mathbb{P}\prt{\abs{X_k(m)}\Ind{m_k<K}\geq \frac{M_n}{m_k}}=0,
\end{equation}
and
\begin{equation}\label{Jam652}
\lim_{n \to \infty} \max_{1 \leq k \leq n}\frac{m_k}{M_n}\mathbb{E}\crt{Y^2_k(m_k)} = 0.
\end{equation}
\end{lemma}
By the union bound, the definition of $Y_k(m_k)$, using
that $\sum_{k} \frac{m_k}{M_n} \leq 1$ we have that
\begin{equation}\label{goodtrunc}
\begin{aligned}
&\limsup_n \mathbb{P}\prt{\bar{S}^K_n \neq \bar{s}^K_n} \leq \limsup_n \sum_{k = 1}^n \mathbb{P}\prt{ X_k(m_k)\Ind{m_k<K}\neq Y_k(m_k)} \\
&= \limsup_n\sum_{k = 1}^n \mathbb{P} \prt{\abs{X_k(m)}\Ind{m_k<K}\geq \frac{M_n}{m_k}} \\
&\leq \limsup_n  \max_{1 \leq k \leq n} \frac{M_n}{m_k} \mathbb{P}\prt{\abs{X_k(m_k)}\Ind{m_k<K}\geq \frac{M_n}{m_k}}\sum_{k=1}^n \frac{m_k}{M_n}\\
&\leq \limsup_n  \max_{1 \leq k \leq n} \frac{M_n}{m_k} \mathbb{P}\prt{\abs{X_k(m_k)}\Ind{m_k<K}\geq \frac{M_n}{m_k}},
\end{aligned}
\end{equation}
the latter can be made arbitrary small
via~\eqref{Jam651}. Hence 
it suffices to consider $\bar{s}^K_n$ instead of
$\bar{S}^K_n$. We next control the mean and the
variance of $\bar{s}^K_n$.
\paragraph{The mean.} As $X_k(m_k)$ is uniformly
integrable family of centred random variables,
by~\eqref{maxto0} it follows that
$\limsup_n \sup_{k} \mathbb{E}\crt{\abs{Y_{k}(m_k)}}=0$, from which it follows that
$\lim_n\mathbb{E}\prt{\bar{s}^K_n}=0$.
\paragraph{The Variance.} Similarly,
by independence and \eqref{Jam652} we can estimate
\begin{equation}
\begin{aligned}\label{2cheby}
  \limsup_n\text{Var}\prt{\bar{s}^K_n}
  & = \limsup_n \sum_{k=1}^n \frac{m_k^2}{M_n^2} \text{Var}(Y_k(m_k))\\
& \leq \limsup_n  \sum_{k=1}^n \frac{m_k}{M_n} \max_{1 \leq k \leq n}\frac{m_k}{M_n}\text{Var}(Y_k(m_k))\\
&\leq \limsup_{n} \max_{1 \leq k \leq n}\frac{m_k}{M_n}\mathbb{E}\crt{Y^2_k(m_k)} = 0.
\end{aligned}
\end{equation}
Finally, $\lim_n\mathbb{E}\prt{\bar{s}^K_n}=0$ together
with \eqref{2cheby} and Chebyshev's inequality yield
\begin{equation}\label{cheby}
\limsup_{n}\mathbb{P} \prt{\bar{s}^K_n \geq \gep} \leq \limsup_n\frac{4}{\gep^2}\text{Var}\prt{\bar{s}^K_n} =0.
\end{equation}	
\qed
\paragraph{Proof of Lemma~\ref{truncY}}
Let $\overline{T}_n : = \inf_{1 \leq k \leq n} \frac{M_n}{m_k}\Ind{m_k<K}$.
Since $\lim_n \overline{T}_n = \infty$,
equation~\eqref{Jam651} follows from~\eqref{W2} as
\begin{equation}\label{factor}
\lim_n  \frac{M_n}{m_k} P\prt{\abs{X_k(m)}\geq \frac{M_n}{m_k}} \leq \lim_n \mathbb{E}\crt{\abs{X_k(m)}\Ind{\abs{X_k(m)}\geq \overline{T}_n}} = 0.
\end{equation}
To prove~\eqref{Jam652},
let $F_{k,m}(a): = \mathbb{P}(\abs{X_k(m)}<a)$ and note
first that integration by parts yields
\begin{equation}\begin{aligned}\label{Iparts}
& \int_0^T x^2 \,dF_{k,m}(x)
= T^2 \mathbb{P}(\abs{X_k(m)}<T) - 2 \int_0^T x \mathbb{P}(\abs{X_k(m)}<x)\, dx\\
&=   T^2 \crt{1 -  \mathbb{P}(\abs{X_k(m)}\geq T)} - 2 \int_0^T x \crt{1-\mathbb{P}(\abs{X_k(m)}\geq x)}\, dx\\
& = -T^2 \mathbb{P}(\abs{X_k(m)}\geq T) +  2 \int_0^T x \mathbb{P}(\abs{X_k(m)}\geq x)\, dx.
\end{aligned}
\end{equation}
Observe further that
\begin{equation}\label{translate}
  \int_0^{\frac{M_n}{m_k}} x^2 \,dF_{k,m}(x)
  = \mathbb{E}\crt{Y^2_k(m_k)},
\end{equation}
and that by the uniform integrability \eqref{W2}
\begin{equation}\label{fgcond}
  \lim_{T \to \infty}\sup_{k,m} T \mathbb{P}(\abs{X_k(m)}\geq T) = 0.
\end{equation}
By \eqref{fgcond} and \eqref{Iparts}, it follows that
\begin{equation}\label{conclu}
  \lim_{n \to \infty} \max_{1 \leq k \leq n}\frac{m_k}{M_n}
  \mathbb{E}\crt{Y^2_k(m_k)} \leq \lim_{n\to \infty}
  \lim_{T \to \infty}\sup_{k \leq n ,m}
  \frac{1}{T}\int_0^T x^2 \,dF_{k,m}(x) = 0.
\end{equation}
\qed

\subsection{Strong law of large numbers for the incremental sum}

\label{sss:incremental}
As in \cite{ACdCdH19}, our proof here relies on an
iterative
scale decomposition into ``small'' and ``big''
increments.
At each scale, the small contribution is defined as the
truncated sum that,
thanks to the stochastic domination
assumption~\eqref{S2}, can be dealt with the techniques
of \cite{Roh71}.
What is left, classified as ``big'', 
is again split (in the next scale) into a ``small'' and
a ``big''.
At this level, the small one is controlled in the same
way as before.
The iteration proceeds until we reach a scale where the
condition \eqref{S1} is sufficient to ensure convergence.
The proof is organized as follows.
We first iteratively decompose the sum $S_n$ into a
finite number of sums of relatively small increments
and one sum of large increments.
Next we show that the large increment sum converges to
zero almost surely.
Finally we prove that each of the small increments also
converge to zero almost surely.
\paragraph{The recursive decomposition.}
We take $\delta$ from~\eqref{S1} and $\gamma$
from~\eqref{S2} and
fix $K = K(\delta,\gamma) \in \mathbb{N}$ such that
\begin{equation}\label{Kgrande}
  \delta K> 1, \quad \text{ and } \quad   \frac{K}{K-1}< 1 + \gamma
\end{equation} 
Now, we let $M^0 := \mathbb{N}$, and let
\begin{equation}\label{i0dec}
  M^{0,s}: = \{j \in M^0 \colon m_j \leq 1\}
\end{equation}
and 
$M^1:=M^0 \setminus M^{0,s}.$
For $i \geq 1$, given $M^i$, if $M^i$ is finite for
$j >i$ set $M^{j},M^{j,s} = \emptyset$, if $M^i$
is an infinite set, let $k^i:\mathbb{N} \to M^i$
be an increasing map with $k^i(\mathbb{N}) = M^i$.
define the $(i)$-st small increments by
\begin{equation}\label{recursivei}
M^{i,s} :=  \chv{k^i_j \in \mathbb{N} \colon  m_{k^{i}_j} < j^{i/K}},
\end{equation}
and define the next level (large) increments
$M^{k+1}:=M^i \setminus M^{i,s}$.
Now,  we let the cardinality of indices
less than $n$ be denoted by
\begin{equation}\label{cardin}
  J(i;n): = \#\chv{j \in M^{i}\colon j \leq n}
  \quad \text{ and }
  \quad
  J(i,s;n): = \#\chv{j \in M^{i,s}\colon j \leq n}.
\end{equation}
To ease the notation set $X_k := X_{k}(m_k)$,
$a^i_{j,n}: = \frac{m_{k^i_j}}{M_n}$, and
$a^{i,s}_{j,n} = \frac{m_{k^{i,s}_j}}{M_n}$.
Since
$
\mathbb{N} = \bigcup_{i=0}^{\Kk} M^{i,s}\cup M^{\Kk},
$
we have that
\begin{equation}\label{e:ssl} 
\begin{aligned}
S_n &= \sum_{i=1}^{\Kk}\underbrace{\sum_{j=1}^{J(i,s;n)} a^{i,s}_{j,n} X_{k^{i,s}_j}} + \underbrace{\sum_{{j=1} }^{J({\Kk};n)}a^{\Kk}_{j,n} X_{k^K_j}}\\
& =\sum_{i=1}^{\Kk}\hspace{1.1cm} S_n^{i,s} \hspace{0.8cm}+ \hspace{1cm}S_n^{\Kk}.
\end{aligned}
\end{equation}
In what follows we show that
\begin{align}
\mathbb{P}\big(\limsup_n S_n^{\Kk} = 0\big)&=1, \label{large0}\\
\mathbb{P}\big(\limsup_n S_n^{i,s} = 0\big)&=1 \text{ for }i\in \{0,1,\ldots, \Kk\}\label{small0}.
\end{align}

\paragraph{The large increments sum.}
To prove \eqref{large0} it is enough to
show that for any $\gep > 0$
\begin{equation}
\mathbb{P}\big(\limsup_n S_n^{\Kk} \leq \gep\big)=1 \label{largeeps}.
\end{equation}
By \eqref{S1}, and the fact that
$m_{k^{\Kk}_j}\geq j^K$, it follows that
\[
\mathbb{P}(X_{k^{\Kk}_j}>\gep) \leq \frac{C_\gep}{(m_{k^{\Kk}_j})^\delta}\leq \frac{C_\gep}{j^{K\delta}}.
\]
Since $K \delta>1$, we conclude that 
$\sum_{j=1}^\infty \mathbb{P}(X_{k^{\Kk}_j}>\gep) <\infty$, 
and by the Borel Cantelli Lemma, we have that
\begin{equation} \label{smallX}
\mathbb{P}\prt{\limsup_j \abs{X_{k^{\Kk}_j} } \leq \gep}=1.
\end{equation}
As $M_n \to \infty$ and $\sum_{j=1}^{J({\Kk},n)}m_{k^{\Kk}_j} \leq M_n$, we conclude that \eqref{largeeps} holds.
\paragraph{The small increment sums.} 
The proof of~\eqref{small0} will be split in two parts,
first we prove it for $i \geq 1$ and then we treat the
case $i = 0$. 
For notation ease, set for any $J \in \mathbb{N}$,
\[ \tilde{m}_j :=m_{k^{i,s}_j} ,
 \quad
\tilde{M}_J : = \sum_{j = 1}^ J \tilde{m}_j,
 \quad
\tilde{a}_{j,J}: = \frac{\tilde{m}_j}{\tilde{M}_J},
 \quad
\text{and let}
 \quad
 \tilde{S}_J = \sum_{j = 1}^ J \tilde{a}_{j,J} X_{k^{i,s}_j}.
\]
Now note that for any $n$
\begin{equation}\label{tilde_i}
S^{i,s}_{n} = \frac{\tilde{M}_{J(i,s;n)}}{M_n}\tilde{S}_{J(i,s;n)}.
\end{equation}
As $\frac{\tilde{M}_{J(i,s;n)}}{M_n} \leq 1$, it follows
that
$\limsup_n \vert{S^{i,s}_n}\vert \leq \limsup_J \vert{\tilde{S}_J}\vert$.
Therefore, it suffices to show that 
\begin{equation}\label{tildeto0}
  \mathbb{P}\prt{\lim_J \tilde{S}_J= 0}=1.
\end{equation}
Since for $i \geq 1$, $m_{k^{i,s}_j}\in [j^{(i-1)/K}, j^{i/K}]$, there are $C,c>0$ for which
\begin{equation}\label{e:bgs} 
\tilde{M}_{J} \geq cJ^{1 +( i-1)/K}, \qquad \tilde{a}_{j,J}\leq \frac{C}{ J^{\frac{K-1}{K} }}.
\end{equation}
Now, as $\lim_J \tilde{a}_{j,J} = 0$, $\sum_{j} \tilde{a}_{j,J} =1$, and conditions \eqref{e:bgs} and \eqref{S2} hold,
one can apply Theorem 2 in~\cite{Roh71} with
$\nu  =\frac{K-1}{K}$ to obtain~\eqref{tildeto0} and
therefore \eqref{small0} for $i \geq 1$. 
To conclude the proof of Theorem~\ref{incrementalstrong}
it remains to verify that $S^{0,s}_n$ converges to $0$
almost surely.
The proof is an adaptation of Theorem 4
in~\cite{JamOrePru65} and is postponed to  Appendix \ref{boundedapp}.

\qed
\subsection{Strong law for the gradual sum}
\label{sss:gradual} 
In this section we prove Theorem~\ref{gradualstrong}. 
Recall the decomposition of $\mathcal{S}_t$
from~\eqref{gradual}.
We note that $\mathcal{S}_t$ is a convex combination of $S_{\ell_t}$ and the boundary term $X_{\ell_t}(\bar{t})$ with  with $\bar{t} = t - M_{\ell_t}$.
By the proof of~\eqref{incrementalstrong},
To prove Theorem~\ref{gradualstrong},
it remains to show that the boundary term vanishes, i.e.
\begin{equation}
\label{goalgradual}
\mathbb{P}\prt{\lim_t \frac{\bar{t}}{t} X_{\ell_t} (\bar{t}) = 0}=1.
\end{equation}
We divide the proof of \eqref{goalgradual} in two steps.
First we show \eqref{goalgradual} for a properly
defined  small  increments, $m_{k+1}< (1 + \alpha_k)M_k$.
Then we show  \eqref{goalgradual} for the complement
set that we refer to as the set  of large increments.
\subsubsection{The small Increments}\label{small_inc}
Let $V_n = \sup \big\{\frac{s}{(M_n + s)}\big \vert X_k(s)\big \vert \colon s \in [0,m_k) \big \}$
and note that
\begin{equation}
  \label{e:sr}
  \limsup_t \frac{\bar{t}}{t}X_{\ell_t}(\bar{t})
  = \limsup_n V_n.
\end{equation}
Thanks to condition~\eqref{S3},
we can control the oscillations $V_n$ for small
increments that satisfy a growth condition defined as follows.
Fix  $\beta>1$ as in \eqref{S3} and let
$\alpha_j = \frac{1}{j^a}$ with $a\in(1/\beta,1)$.
The first small increment is defined by
\begin{equation}\label{small_1}
k_1': = \inf\{k \in \mathbb{N} \colon M_{k+1} < (1 + \alpha_1) M_k\},
\end{equation}
and define recursively for $j$-th small increment by
\begin{equation}\label{small_j}
k'_{j+1} := \inf\chv{k \in \mathbb{N} \colon k > k'_j, \quad M_{k+1} < (1 + \alpha_{j+1}) M_k } .
\end{equation}
If for some $j, k_{j} = \infty$ this implies there are
only finitely many small increments and we do not need
to worry about them in \eqref{e:sr}.
If for all $j, k_{j}' < \infty$, we claim that almost
surely
\begin{equation}
  \label{e:srstar}
  \limsup_j V_{k'_j} = 0.
\end{equation}
Indeed, as $m_{k'_j + 1}< \alpha_j M_{k'_j}$,
by \eqref{S3}, with $r= 0$,  it follows that for any
$\gep>0$
\begin{equation}
  \label{e:nod}
\mathbb{P}(V_{k'_j}> \gep ) \leq \mathbb{P}\crt{\sup_{s\leq \alpha_j M_{k'_j}} s\vert X_{k'_j}(s)\vert >\gep M_{k'_j} }\leq \alpha_j^\beta C_\gep .
\end{equation}
As $\sum_j \alpha_j^\beta< \infty$, by the
Borel-Cantelli lemma we conclude that
\begin{equation}
\label{e:ndc} 
 \quad\mathbb{P} \prt{\limsup_j  V_{k'_j} \leq \gep} = 1,
\end{equation}
and since $\gep>0$ is arbitrary, it follows that
\begin{equation}
  \label{e:ndc0} 
 \mathbb{P} \prt{\lim_j  V_{k'_j} =0} = 1.
\end{equation}
\qed
\subsubsection{The large increments}\label{large_inc}
By \eqref{e:ndc0} we can restrict our attention to
$\chv{k^*_1,k^*_2,\ldots}= \mathbb{N}\setminus \chv{k'_1, k'_2, \ldots}$.
Note that since $\alpha_j \leq 1$
\begin{equation}\label{lbalpha}
(1 + \alpha_j) \geq C \exp(\alpha_j/2),
\end{equation}
for some $C>0$.
Therefore, for some $c_a>0$ the following growth
condition holds for the terms in the sequence
$\chv{M_{k_i^*}}_i$
\begin{equation}\label{e:grc}
  M_{k_{i}^*}\geq \prod_{j=1}^i(1 + \alpha_j)M_1
  \geq C\exp(\sum_{j=1}^i\frac{ \alpha_j}{2})M_1
  \geq \exp(c_ai^{1-a})M_1.
\end{equation}
The proof now proceeds in two steps, we first show
that the boundary term
$\frac{\bar{t}}{t} X_{\ell_t}(\bar{t})$
converges to zero along a subsequence
$\big\{t_{i,j}, i,j \in \mathbb{N}\cup\{0\}\big\}$,
what we call \emph{pinning}, and then based on this
result we show that the full sequence converges to zero
as we bound its oscillations on the intervals
$[t_{i,j} , t_{i,j+1}]$.  
\paragraph{Pinning.}
For the  boundary increments
$k \in \chv{k^*_1,k^*_2,\ldots}$ consider the
following pinning procedure.
Let $k^*_0 = i(k^*_0) = 0$
and define recursively for $n\in \mathbb{N}$
\[
i(k^*_{n}):= \inf \{i> i(k^*_{n-1}) \colon \prod_{j = i(k^*_{n-1})}^i (1 + \alpha_j)M_{k^*_n} > M_{k^*_n + 1}\}.
\]
We note that \eqref{lbalpha} and
$\sum_j \alpha_j = \infty$ imply that
$i(k^*_n)< \infty$ for all $n$.
Now to define the pinning sequence let
$t_{i,0}: = k^{*}_{i}$ and for
$j\in \{0 , \ldots, i(k^*_i) - i(k^{*}_{i-1})\}$
set
\begin{equation}\label{e:psec}
  t_{i,j} :=
  \begin{cases}
    (1+\alpha_{i(k^*_{i-1})+ j})t_{i,j-1} & \text{ if } j < i(k^*_i)- i(k^{*}_{i-1})\\
   M_{k^*_i + 1} & \text{ if } j =  i(k^*_i)- i(k^{*}_{i-1}).
  \end{cases}
\end{equation}
Now it follows from the definition of $\bar{t}$ that 
\begin{equation}\label{e:pig}
  \bar{t}_{i,j} =  t_{i,j} - M_{k^*_i} = M_{k^*_i}[\prod_{n =1}^{j}(1 + \alpha_{i(k^*_{i-1})+n})-1].
\end{equation}
By the polynomial decay  in \eqref{S1} it follows that
for any $\gep>0$, and $i,j>0$
\begin{equation}\label{e:bp}
\mathbb{P}\crt{\abs{\frac{\bar{t}_{i,j}}{t_{i,j}}X_{k^*_i}(\bar{t}_{i,j})} \geq \gep} \leq \mathbb{P}\crt{\abs{X_{k^*_i}(\bar{t}_{i,j})} \geq \gep} \leq \frac{C_\gep}{\prt{\bar{t}_{i,j}}^\delta}.
\end{equation}
By~\eqref{e:pig} and~\eqref{e:grc}, the sum over
$i,j>0$ of the above probability is finite and
therefore for any $\gep >0$
\begin{equation}\label{e:pin}%
  \mathbb{P}\crt{\frac{\bar{t}_{i,j}}{t_{i,j}}  \abs{X_{k^*_i}(\bar{t}_{i,j})} \geq \gep \text{ for infinitely many } (i,j)} = 0,
\end{equation}
which implies that
\begin{equation}\label{e:pin0}%
\mathbb{P}\crt{\limsup_{i,j} \frac{\bar{t}_{i,j}}{t_{i,j}}    \abs{X_{k^*_i}(\bar{t}_{i,j})}=0} = 1.
\end{equation}
It remains to understand the behaviour of the boundary
term in $[t_{i,j}, t_{i,j+1}]$.
\paragraph{Oscillations.}
Now we use~\eqref{S3} to compute the oscillations
between the pinned values of the boundary.
Fix $\gep>0$ and consider the event $\Omega_{i_0}$
defined by 
\begin{equation}\label{pinned}
\Omega_{i_0}:= \Big\{\sup_j\abs{\frac{\bar{t}_{i,j}}{t_{i,j}}X_{k^*_i}\prt{\bar{t}_{k^*_i,j}}} 
\leq \gep, \quad \text{ for } i >i_0 \Big\}.
\end{equation}
Therefore, on $\Omega_{i_0}$ for $t \in [t_{i,j}, t_{i,j+1}]$, and   $j \geq 1$
\begin{align*}
&\abs{\frac{\bar{t}}{t} X_{k^*_i}\prt{\bar{t}} -\frac{\bar{t}_{i,j}}{t_{i,j}} X_{k^*_i}\prt{\bar{t}_{i,j}}}\\
&= \abs{\frac{1}{t}\crt{\bar{t}  X_{k^*_i}\prt{\bar{t}} -  \bar{t}_{i,j} X_{k^*_i}\prt{\bar{t}_{i,j}}} + \prt{\frac{\bar{t}}{t} - \frac{\bar{t}_{i,j}}{t_{i,j}}}X_{k^*_i}\prt{\bar{t}_{i,j}}}\\
&\leq \frac{1}{t}\abs{{\bar{t}  X_{k^*_i}\prt{\bar{t}} -  \bar{t}_{i,j} X_{k^*_i}\prt{\bar{t}_{i,j}}}}
+ (C_\alpha-1) \gep
\end{align*}
where 
\begin{equation}
  C_\alpha = \sup_{i,j}\sup_{t \in [t_{i,j}, t_{i,j+1}]}\frac{\frac{\bar{t}}{t}}{\frac{\bar{t}_{i,j}}{t_{i,j}} }
  < \infty 
\end{equation}
Let $s: = t - t_{i,j}$. By \eqref{S3} it follows that
on $\Omega_{i_0}$
\begin{equation}
\label{e:boit}
\mathbb{P} \crt{\sup_{s \leq t_{i,j+1} - t_{i,j} }
  \abs{\frac{\bar{t}_{i,j} + s}{t_{i,j} + s}X_{k^*_i}{(\bar{t}_{i,j} + s)}-\frac{\bar{t}_{i,j}}{t_{i,j}}X_{k^*_i}{(\bar{t}_{i,j})}}>C_\alpha\gep}
\leq \alpha^\beta_{i,j}C_\gep.
\end{equation}
As the sum of the above terms over $i, j\in \mathbb{N}$
is finite, by \eqref{pinned}, it follows that 
\begin{equation}\label{e:bcit}
  \mathbb{P} \crt{\limsup_k V_k \leq C_\alpha\gep}  \geq  \limsup_{i_0} \mathbb{P}(\Omega_{i_0}) = 1
\end{equation}
Since $\gep >0$ is arbitrary, from~\eqref{e:sr},
\eqref{e:ndc} and~\eqref{e:bcit}
we conclude that~\eqref{goalgradual} holds.
\qed
\appendix
\section{Bounded increments}\label{boundedapp}
To deal with the case $i = 0$, in the case of small
incrementys defined in \eqref{i0dec},
if $\lim_{n}\sum_{i=1}^n m^{i,0}_n<\infty$ it follows
that $S^{i,0}_n $ converges to $0$.
For this reason assume without loss of generality that
$m_k = m^{i,0}_k$ and that
\begin{equation}\label{diverge0s}
\lim_n M_n =\lim_n \sum_{i=1}^n m^{i,0}_n \to \infty.  
\end{equation}
We next consider the truncated versions of $X_k$
\begin{equation}\label{truncatedY}
  Y_k: = X_k \Ind{\big\{\vert{X_{k}\vert}\leq \frac{M_k}{m_k}\big \}}.
\end{equation}
The proof proceeds in two steps: first we show that
\begin{equation}\label{eventualequal}
  \mathbb{P} \prt{Y_{k} \neq X_{k} \text{ i.o.} } = 0.
\end{equation}
This implies that  the limit of $S_n $ equals the limit
of $\bar{S}_n: = \sum_{k = 1}^n a_{n,k}Y_{k}$.
The proof will be complete once we prove that
\begin{equation}\label{limit_trunc}
  \mathbb{P} \prt{\lim_n\bar{S}_n=0 } = 1.
\end{equation}
\paragraph{Proof of \eqref{eventualequal}.}
Let $ N(x): = \chv{k \colon\frac{M_k}{m_k} \leq x }$,
$F^*(a): = \mathbb{P}(\abs{X^*}<a)$, and note that by
the stochastic domination \eqref{S2}
\begin{equation}
  \begin{aligned}
    \sum_{k} \mathbb{P} \prt{Y_{k} \neq X_{k}}
    &\leq \sum_k \mathbb{P}\prt{\abs{X_{k}}
      \geq \frac{M_k}{m_k} }\leq \sum_j \mathbb{P}\prt{\abs{X^*}
      \geq \frac{M_k}{m_k} }\\
    & \leq \sum_j \int_{x \geq \frac{M_k}{m_k}} \,d F^*(x)= \int N(x) \,dF^*(x) = \mathbb{E}\crt{N(\abs{X^*})}.
  \end{aligned}
\end{equation}
To obtain \eqref{eventualequal} it remains to prove that
\begin{equation}\label{remainsN}
  \mathbb{E}\crt{N(\abs{X^*})}<\infty.
\end{equation}
This step follows from Lemma 2
of~\cite{JamOrePru65} which states that
\begin{equation}\label{limsupN}
  \limsup \frac{N(x)}{x \log x}\leq 2.
\end{equation}
By~\eqref{limsupN},
it follows that $N(x) \leq C x^{1 + \gamma}$ and
by~\eqref{S2}, $\mathbb{E}\crt{N(\abs{X^*})} < \infty$.
\qed
\paragraph{Proof of \eqref{limit_trunc}}
As
\[
\lim_n \mathbb{E}\crt{\bar{S}_n} = 0,
\]
to prove \eqref{limit_trunc} it suffices to show that
\begin{equation}\label{varvanish}
\sum_k \frac{m^2_k}{M^2_k} \text{Var}(Y_{k})  < \infty.
\end{equation}
As $\frac{M_k}{m_k} \to \infty$ it follows that there
is a $C$ such that
\[
  \mathbb{E}\crt{Y_k^2} \leq  C \int_{\abs{x}\leq \frac{M_k}{m_k}} x^2 \, d F^*(x).
\]
Therefore, the sum in \eqref{varvanish} can be bounded
by
\begin{equation}
  C \sum_k \frac{m^2_k}{M^2_k}\int_{\abs{x}<\frac{M_k}{m_k}}x^2 \, dF^*(x)
  = C\int x^2 \sum_{k \colon \frac{M_k}{m_k}\geq \abs{x}}\frac{m^2_k}{M^2_k} d F^*(x).
\end{equation}
To complete the proof it remains to show that the
right-hand side above is finite.
This follows from the following claims whose proofs are
given right after:
\begin{equation}\label{claim1}
  \sum_{k \colon \frac{M_k}{m_k}\geq \abs{x}}\frac{m^2_k}{M^2_k}
  \leq 2 \int_{y \geq \abs{x}}\frac{N(y)}{y^3}\, dy,
\end{equation}
and 
\begin{equation}\label{claim2}
\int x^2 \int_{y \geq \abs{x}}\frac{N(y)}{y^3}\, dy\, dF^*(x)< \infty.
\end{equation}
\qed
\paragraph{proof of \eqref{claim2}}
By \eqref{limsupN}
it follows that there are $C>0$ and $\gamma \in (0,1)$
such that $N(x)\leq Cx^{1 + \gamma}$. Therefore
\begin{equation}
\begin{aligned}  
&\int x^2 \int_{y \geq \abs{x}}\frac{N(y)}{y^3}\, dy\, dF^*(x)\leq  \int x^2 \int_{y \geq \abs{x}}\frac{Cy^{1 + \gamma}}{y^3}\, dy\, dF^*(x)\\
& = \int x^2 \int_{y \geq \abs{x}}\frac{C}{y^{2 - \gamma}}\, dy\, dF^*(x) = \int x^{2} \frac{C}{(1 - \gamma)x^{1 - \gamma}}\, d F^*(x)\\
& = \frac{C}{1 - \gamma}\mathbb{E}\crt{\abs{X^*}^{1 + \gamma}}< \infty.
\end{aligned}  
\end{equation}    
\qed           

\paragraph{proof of \eqref{claim1}}
Observe that by the definition of $N$ and integration
by parts
\begin{equation}
\begin{aligned}
\sum_{k \colon \abs{x}< \frac{M_k}{m_k}\leq z} \frac{m^2_k}{M^2_k}&= \int_{\abs{x}<y <z} \frac{dN(y)}{y^2}\\
&= \frac{N(z)}{z^2} - \frac{N(\abs{x})}{x^2}
+ 2 \int_{\abs{x}<y<z}\frac{N(y)}{y^3}\, dy.
\end{aligned}
\end{equation}
Furthermore, since $N(z) \leq N(y)$ for $z \leq y$ and
$\frac{1}{z^2} = 2\int_z^\infty \frac{1}{y^3}\, dy$ it
follows from \eqref{claim2} that 
\[
\frac{N(z)}{z^2} \leq 2\int_z^\infty \frac{N(z)}{y^3}\, dy \to 0,
\]
and so
\begin{equation}
  \begin{aligned}
    \sum_{k \colon \abs{x}< \frac{M_k}{m_k}} \frac{m^2_k}{M^2_k}
    &= \lim_z \sum_{k \colon \abs{x}< \frac{M_k}{m_k}<z} \frac{m^2_k}{M^2_k}
    \int_{\abs{x}<y <z} \frac{dN(y)}{y^2}\\
    &\leq  2 \int_{\abs{x}<y}\frac{N(y)}{y^3}\, dy.
  \end{aligned}
\end{equation}  
\qed
\section*{Acknowledgements}
The research in this paper was partially supported
through NWO Gravitation Grant\\
\mbox{NETWORKS-024.002.003}.
We are grateful to Evgeny Verbitsky for pointing the
paper of Rohatgi~\cite{Roh71}
and related useful discussions.    

\bibliographystyle{plain}
\bibliography{mybib.bib}      

\begin{thebibliography}{10}

\bibitem{ACdCdH19}
Luca Avena, Yuki Chino, Conrado da~Costa, and Frank den Hollander.
\newblock Random walk in cooling random environment: ergodic limits and
  concentration inequalities.
\newblock {\em Electron. J. Probab.}, 24:1--35, 2019.

\bibitem{ACdCdH20}
Luca {Avena}, Yuki {Chino}, Conrado {da Costa}, and Frank {den Hollander}.
\newblock {Random walk in cooling random environment: recurrence versus
  transience and mixed fluctuations}.
\newblock {\em To appear in AIHP, arXiv:1903.09200}, 2019.

\bibitem{AdH19}
Luca Avena and Frank~den Hollander.
\newblock Random walks in cooling random environments.
\newblock In Vladas Sidoravicius, editor, {\em Sojourns in Probability Theory
  and Statistical Physics - III}, pages 23--42, Singapore, 2019. Springer
  Singapore.

\bibitem{Ber1713}
Jacob Bernoulli.
\newblock {\em Ars Conjectandi, Opus Posthumum}.
\newblock Basileae, Impensis Thurnisiorum, Fratrum Werke 3, 1713.

\bibitem{Bin15}
N.~H. Bingham.
\newblock {\em Riesz Means and Beurling Moving Averages}, chapter~8, pages
  159--172.
\newblock World Scientific Europe, 2019.

\bibitem{Bin17}
N.H. Bingham and Bujar Gashi.
\newblock Voronoi means, moving averages, and power series.
\newblock {\em Journal of Mathematical Analysis and Applications},
  449(1):682--696, 2017.

\bibitem{Dur19}
Rick Durrett.
\newblock {\em Probability: Theory and Examples}.
\newblock Cambridge Series in Statistical and Probabilistic Mathematics.
  Cambridge University Press, 5 edition, 2019.

\bibitem{FriVel17}
Sacha Friedli and Yvan Velenik.
\newblock {\em Statistical Mechanics of Lattice Systems: A Concrete
  Mathematical Introduction}.
\newblock Cambridge University Press, 2017.

\bibitem{JamOrePru65}
Benton Jamison, Steven Orey, and William Pruitt.
\newblock Convergence of weighted averages of independent random variables.
\newblock {\em Z. Wahrscheinlichkeitstheorie und Verw. Gebiete}, 4(1):40--44,
  Mar 1965.

\bibitem{Kle14}
Oleg Klesov.
\newblock {\em Limit Theorems for Multi-Indexed Sumsof Random Variables},
  volume~71 of {\em Probability Theory and Stochastic Modelling}.
\newblock Springer, Berlin, Heidelberg, 2014.

\bibitem{Pru66}
William Pruitt.
\newblock Summability of independent random variables.
\newblock {\em J. Math. and Mech.}, 15(5):769--776, 1966.

\bibitem{Roh71}
Vijay Rohatgi.
\newblock Convergence of weighted sums of independent random variables.
\newblock {\em Mathematical Proceedings of the Cambridge Philosophical
  Society}, 69(2):305–307, 1971.

\bibitem{Rev67}
Pál Révész.
\newblock {\em Laws of Large Numbers}.
\newblock Academic Pres, 1 edition, 1967.

\bibitem{Sto68}
William Stout.
\newblock Some results on the complete and almost sure convergence of linear
  combinations of independent random variables and martingale differences.
\newblock {\em Ann. Math. Statist.}, 39(5):1549--1562, 10 1968.

\bibitem{Sulet05}
Matthew Sullivan, Steven Pirotta, V.~Chernenko, GH~Wu, G~Balasubramanium, Susan
  Hua, and Harsh Chopra.
\newblock Magnetic mosaics in crystalline tiles: The novel concept of
  polymagnets (invited).
\newblock {\em International Journal of Applied Electromagnetics and
  Mechanics}, 22:11--23, 10 2005.

\end{thebibliography}
\end{document}